\documentclass{amsart}

\usepackage[T1]{fontenc}
\usepackage[utf8x]{inputenc}
\usepackage{lmodern, euscript}
\usepackage{enumerate}
\usepackage{mathabx}
 \usepackage[all,cmtip]{xy}

\usepackage{amsfonts,amsmath,amstext,amsbsy,amssymb,amsbsy,
  amsopn,amsthm,amscd} 
\usepackage[T1]{fontenc}
\usepackage{hyperref}
\usepackage{mathrsfs, graphicx}

\RequirePackage[dvipsnames]{xcolor} 
\definecolor{halfgray}{gray}{0.55} 
\definecolor{webgreen}{rgb}{0,0.5,0}
\definecolor{webbrown}{rgb}{.6,0,0} \hypersetup{%
  colorlinks=true, linktocpage=true, pdfstartpage=3,
  pdfstartview=FitV,%
  breaklinks=true, pdfpagemode=UseNone, pageanchor=true,
  pdfpagemode=UseOutlines,%
  plainpages=false, bookmarksnumbered, bookmarksopen=true,
  bookmarksopenlevel=1,%
  hypertexnames=true,
  pdfhighlight=/O,
  urlcolor=webbrown, linkcolor=RoyalBlue,
  citecolor=webgreen, 
  pdftitle={Recap},%
  pdfsubject={},%
  pdfkeywords={},%
  pdfcreator={pdfLaTeX},%
  pdfproducer={LaTeX with hyperref}%
}

\newtheorem{theorem}{Theorem}[section]
\newtheorem{add}[theorem]{Addendum}
\newtheorem{lemma}[theorem]{Lemma}
\newtheorem{corollary}[theorem]{Corollary}
\newtheorem{proposition}[theorem]{Proposition}
\theoremstyle{definition}
\newtheorem{problem}[theorem]{Problem}

\newtheorem{remark}[theorem]{Remark}
\newtheorem*{remarks}{Remarks}

\newcommand{\field}[1]{\mathbb{#1}}
\newcommand{\R}{\field{R}}

\newcommand{\C}{\field{C}}
\newcommand{\Z}{\field{Z}}
\newcommand{\Q}{\field{Q}}
\newcommand{\T}{\field{T}}

\newcommand{\cD}{\mathcal{D}}

\newcommand{\cF}{\mathcal{F}}

\newcommand{\cH}{\mathcal{H}}

\newcommand{\cU}{\mathcal{U}}
\newcommand{\cV}{\mathcal{V}}

\newcommand{\cT}{\EuScript T}
\newcommand{\euP}{\EuScript P}

\newcommand{\ie}{{\it i.e., } }
\newcommand{\eg}{{\it e.g., } }

\renewcommand{\phi}{\varphi}

\newcommand{\eps}{\varepsilon}

\renewcommand{\|}{\,\Vert\,}

\begin{document}
\baselineskip=14pt

\title[Rigidity for very non-algebraic Anosov diffeomorphisms]{Smooth rigidity for very non-algebraic Anosov diffeomorphisms of codimension one}
\author {Andrey Gogolev and Federico Rodriguez Hertz}\thanks{The authors were partially supported by NSF grants DMS-1955564 and DMS-1900778, respectively}

 \address{Department of Mathematics, The Ohio State University,  Columbus, OH 43210, USA}
\email{gogolyev.1@osu.edu}

\address{Department of Mathematics, The Pennsylvania State University,  University Park, PA 16802, USA}
\email{hertz@math.psu.edu}

\begin{abstract} 
  \begin{sloppypar}
In this paper we introduce a new methodology for smooth rigidity of Anosov diffeomorphisms based on ``matching functions.'' The main observation is that under certain bunching assumptions on the diffeomorphism the periodic cycle functionals can provide such matching functions. For example we consider a sufficiently small $C^1$ neighborhood of a linear hyperbolic automorphism of the 3-dimensional torus which has a pair of complex conjugate eigenvalues. Then we show that two very non-algebraic (an open and dense condition) Anosov diffeomorphisms from this neighborhood are smoothly conjugate if and only they have matching Jacobian periodic data. We also obtain a similar result for certain higher dimensional codimension one Anosov diffeomorphisms. 
  \end{sloppypar}
\end{abstract}
\maketitle

\section{Introduction}

Recall that a diffeomorphism $f\colon M\to M$ is called {\it Anosov} if the tangent bundle admits a $Df$-invariant splitting $TM=E^s\oplus E^u$, where $E^s$ is uniformly contracting and $E^u$ is uniformly expanding under $f$. Basic examples of Anosov diffeomorphisms are {\it toral hyperbolic automorphisms} $L\colon \T^d\to\T^d$ which are given by {\it hyperbolic matrices} in $GL(d,\Z)$, \ie matrices whose spectrum is disjoint with the unit circle in $\C$. 

Let $f_1, f_2\colon M\to M$ be transitive Anosov diffeomorphisms which are conjugate via a homeomorphism $h$, $h\circ f_1=f_2\circ h$. We will say that $f_1$ and $f_2$ have matching {\it periodic data} if for every periodic point $p=f_1^k(p)$ the differentials $(Df_1^k)_p$ and $(Df_2^k)_{h(p)}$ are conjugate (in particular, they have the same spectrum). By differentiating the conjugacy relation one immediately sees that matching of periodic data is a necessary assumption for the conjugacy to be $C^1$. A weaker assumption which we will consider here is matching of Jacobian periodic data. Namely, we say that $f_1$ and $f_2$ have matching {\it Jacobian periodic data} if every periodic point $p=f_1^k(p)$
$$
(J^sf^k_1)_p=(J^sf^k_2)_{h(p)}\,\,\,\,\,\mbox{ and}\,\,\, (J^uf_1^k)_p=(J^uf_2^k)_{h(p)} 
$$
where $J^sf_i$ and $J^uf_i$ stand for Jacobians of the restrictions of $Df_i$, $i=1,2$, to the stable and unstable distributions, respectively. If the equality holds only for stable (or only for unstable) Jacobians then we will talk about matching of {\it stable (respectively, unstable) Jacobian periodic data}.

In dimension 2 matching of periodic data implies smoothness of the conjugacy by works of de la Llave, Marco and Moriy\'on~\cite{dlL0, MM, dlL}. In higher dimensions a lot of work was devoted to periodic data rigidity (characterization of smooth conjugacy class) of hyperbolic automorphisms, see \eg~\cite{dlL2, KS2, GKS, DW}. In particular, in dimension 3 the problem was solved for automorphisms with a pair of complex eigenvalues by Kalinin and Sadovskaya~\cite{KS2} and for automorphisms with real spectrum by Gogolev and Guysinsky~\cite{GG, G2}. Further, in proximity of automorphism with real spectrum matching of periodic data implies $C^{1+\mathrm{h\"older}}$ regularity of the conjugacy on an open set of Anosov diffeomorphisms in dimension 3 and higher~\cite{GG, G}.

In this paper we transfer some of ideas of~\cite{GRH} from the setting of expanding maps to the setting of Anosov diffeomorphisms. In particular, we have open sets of Anosov diffeomorphisms where we obtain optimal smoothness of the conjugacy using less data, such as Jacobian periodic data, or stable Jacobian periodic data as opposed to full periodic data which was commonly used before. To the best of our knowledge the only prior result when $C^\infty$ smooth conjugacy classes were characterized on an open set of diffeomorphisms in dimension $>2$ is the work of Palis and Yoccoz~\cite{PY} which gave a complete set of invariant of smooth conjugacy for Morse-Smale diffeomorphisms.\footnote{It is a classical KAM theorem of Arnold and Moser that the local smooth conjugacy class of a Diophantine translation on $\T^d$ is characterized by the rotation vector. Further, the Arnold-Moser normal form~\cite[Th\'eor\`eme 2.2]{Her} implies that this smooth conjugacy class is codimension $d$ submanifold in $\textup{Diff}_0(\T^d)$. However it is hopeless to characterize smooth conjugacy classes in an open neighborhood of the Diophantine translation since even topological conjugacy classes are not understood.}

Our proof is based on matching functions technique. Namely, we look at pairs of $C^r$ functions $(\Phi_1,\Phi_2)$ which are defined on the local unstable leaves $W^u_{f_1,loc}(x)$ and $W^u_{f_2,loc}(h(x))$ and which are {\it matching} in the sense that they satisfy the matching relation
$$
\Phi_1=\Phi_2\circ h
$$
Roughly speaking, the idea of the proof is to find sufficiently many independent matching pairs so that we can apply the inverse function theorem and conclude that $h$ is $C^r$ when restricted to the local unstable leaf.

 In~\cite{GRH} the source of matching functions was given by normalized matching potentials (logarithms of Jacobians), where normalizations came from transfer operators associated to the expanding maps. In this paper we explore different matching functions which are given by periodic cycle functionals. Periodic cycles functionals were originally introduced by Katok and Kononenko~\cite{KK} in order to study the cohomological equation over partially hyperbolic diffeomorphisms.

\subsection{Results in dimension 3} We present several results for Anosov diffeomorphisms in dimension 3.

\begin{theorem}
\label{thm_main}
Let $L\colon\T^3\to\T^3$ be an Anosov automorphism with a pair of complex conjugate eigenvalues. Then there exists a $C^1$ neighborhood $\cU$ of $L$ such that all $C^{r}$, $r\ge 2$, Anosov diffeomorphisms $f_1, f_2\in\cU$ with matching Jacobian periodic data  are either $C^{r_*}$ conjugate or the SRB measure coincides with the measure of maximal entropy for $f_1$.
\end{theorem}
Above $r_*=r$ if $r$ is not integer and $r_*=(r-1)+Lip$ if $r$ is an integer.
Note that there is a unique topological conjugacy $h$ which is $C^0$ close to $id_{\T^3}$, $h\circ f_1=f_2\circ h$, coming from structural stability. The condition on matching of Jacobian periodic data is imposed relative to this conjugacy $h$.
\begin{remarks} 
\begin{enumerate}
\item
Recall that SRB and MME measures coincide if and only $-\log J^uf_1$ is cohomologous to a constant~\cite{B}. By the Livshits theorem~\cite{L}, we know that $-\log J^uf_1$ is not cohomologous to constant if and only if there are two periodic points with different unstable Lyapunov exponents. Hence the property of SRB measure being different from MME measure can be detected from two periodic points with different  unstable Lyapunov exponents. Hence the above theorem solves the smooth rigidity problem in a $C^1$ neighborhood of $L$ on a $C^1$-open and $C^\infty$-dense subset. The obvious remaining problem is to handle the case when $-\log J^uf_1$ is cohomologous to a constant. It is not hard to see by perturbing $L$  along unstable foliation that the conjugacy is not necessarily smooth if we only assuming matching of Jacobian periodic data. However, the problem that remains in this case is establishing smoothness of the conjugacy under the assumption of matching of (full) periodic data. 
\item We can replace the assumption on $C^1$-closeness to $L$ by an appropriate bunching assumption and existence of a periodic points with complex conjugate eigenvalues.
\item If both $f_1$ and $f_2$ are volume preserving then it is enough to assume matching of unstable Jacobian periodic data because the stable Jacobian periodic data are given by reciprocals and, hence, match automatically.
\end{enumerate}
\end{remarks}

 Given two conjugate Anosov diffeomorphisms $f_1$ and $f_2$, $h\circ f_1=f_2\circ h$, and H\"older continuous functions $\phi_1, \phi_2:M\to\R$, we say that $(f_1,\phi_1)$ is {\it equivalent} to $(f_2, \phi_2)$ and write 
$$(f_1,\phi_1)\sim (f_2,\phi_2)$$
 if there exists a continuous function $u:M\to\R$ such that 
$$\phi_1-\phi_2\circ h=u-u\circ f_1$$
Then, by the Livshits theorem~\cite{L}, $(f_1,\phi_1)\sim (f_2,\phi_2)$ if and only if
for every periodic point  $x\in Fix(f_1^n)$
	$$
	\sum_{k=0}^{n-1}\phi_1(f_1^k(x))=\sum_{k=0}^{n-1}\phi_2(f_2^k(h(x)))
	$$
	
Also recall that a potential $\phi\colon M\to\R$ is called an {\it almost coboundary} over $f\colon M\to M$ if $\phi$ is cohomologous to a constant, that is,
$$
\phi=u-u\circ f+c
$$
for some function $u$ and a constant $c$.

In fact, when $f_1$ and $f_2$ are at least $C^3$, Theorem~\ref{thm_main} is a consequence of the following more general result.
\begin{theorem}
\label{thm_tech}
Let $L\colon\T^3\to\T^3$ be an Anosov automorphism with a pair of complex conjugate eigenvalues whose absolute value is greater than 1. Fix a number $\kappa\in(\frac12,1]$. Then there exists a $C^1$ neighborhood $\cU$ of $L$ such that if $(f_1,\phi_1)\sim (f_2,\phi_2)$, where $f_1,f_2\in\cU$ are $C^r$, $r\ge2+\kappa$, and $\phi_1,\phi_2\in C^{1+\kappa}(\T^3)$, then either $h$ is uniformly $C^r$ along unstable leaves or $\phi_1$ is an almost coboundary over $f_1$.
\end{theorem}

\begin{corollary}
\label{cor1}
Let $L\colon\T^3\to\T^3$ be an Anosov automorphism with a pair of complex conjugate eigenvalues of absolute value $>1$. Then there exists a $C^1$ neighborhood $\cU$ of $L$ such that all $C^{r}$, $r\ge 3$, Anosov diffeomorphisms $f_1, f_2\in\cU$ with matching stable Jacobian periodic data  are either $C^{r_*}$ conjugate or the SRB measure coincides with the measure of maximal entropy for $f_1^{-1}$.
\end{corollary}

The next corollary established smooth conjugacy only assuming matching of full Jacobian periodic data in the dissipative setting.

\begin{corollary} 
\label{cor2}
Let $L\colon\T^3\to\T^3$ be an Anosov automorphism with a pair of complex conjugate eigenvalues. Fix a number $\kappa\in(\frac12,1]$.  Then there exists a $C^1$ neighborhood $\cU$ of $L$ such that if $f_1$ is not volume preserving and $(f_1,\log J f_1)\sim (f_2,\log Jf_2)$, where $f_1,f_2\in\cU$ are $C^r$, $r\ge2+\kappa$, and $J f_i$ is full Jacobian with respect to a fixed volume form, $i=1,2$, then $f_1$ and $f_2$ are $C^{r_*}$ conjugate.
\end{corollary}
To the best of our knowledge this is the first result which only uses full Jacobians at periodic points. Analogous statement is unknown in dimension 2 and cannot be proved with the method of this paper.

Our approach is also useful in the case of real spectrum.
Namely, we can partially improve results of~\cite{GG} by bootstrapping regularity of the conjugacy under an additional bunching assumption. 

\begin{theorem} 
\label{thm_misha}
Let $L\colon\T^3\to\T^3$ be a hyperbolic automorphisms with real spectrum $\{\mu^{-1},\lambda,\lambda^\alpha\}$, where $0<\mu^{-1}<1<\lambda<\lambda^\alpha$. Assume that $\alpha<\frac{1+\sqrt{17}}{4}$. Then there exists a sufficiently small $C^1$ neighborhood $\cU$ of $L$ and an open dense subset $\cV\subset\cU$ such that if smooth Anosov diffeomorphisms $f_1, f_2\in\cV$ have the same Jacobian periodic data then they are smoothly conjugate.
\end{theorem}

The set $\cV$ will be described explicitly in the course of the proof, in particular, $L\notin\cV$. It is interesting that this brings us rather close to full smooth classification in the neighborhood of $L$. For full classification one of course would need to assume matching of complete periodic data not only Jacobians. Still some serious difficulties remain, especially in the case when both stable and unstable Jacobians are cohomologous to constants. However, the case when all three Jacobians, including strong unstable are cohomologous to constants was already handled in~\cite{G2}. We repeat here a restricted version of a conjecture from~\cite{G2}.

\begin{problem} In the setting of the above theorem prove that if $f_1,f_2\in\cU$ have matching periodic data then they are smoothly conjugate.
\end{problem}


\subsection{A codimension one result}

We will call a toral automorphism $L\colon\T^d\to\T^d$ {\it generic} if
\begin{enumerate}
\item $L$ is hyperbolic, \ie the spectrum of $L$ is disjoint with the unit circle in $\C$;
\item $L$ is irreducible, \ie its characteristic polynomial is irreducible over $\Q$;
\item no three eigenvalues of $L$ have the same absolute value;
\item if two eigenvalues of $L$ have the same absolute value then they are a pair of complex conjugate eigenvalues.
\end{enumerate}

Local $C^{1+\textup{h\"older}}$ conjugacy class of a generic automorphism $L$ was characterized in terms of periodic data by Gogolev-Kalinin-Sadovskaya~\cite{GKS}. They also proved that ``most'' automorphisms of $\T^d$ are generic, that is, the proportion of non-generic automorphisms of $\T^d$ goes to $0$ as $\|L\|\to+\infty$. The following result provides $C^{1+\textup{h\"older}}$ conjugacy for very non-algebraic diffeomorphisms in a neighborhood of $L$ assuming coincidence of Jacobian periodic data. However, we need to impose further restrictions on $L$ --- codimension one and a bunching condition.

\begin{theorem}
\label{thm_codim1}
Let $L\colon\T^d\to \T^d$ be a generic automorphism with one dimensional stable subspace. Assume that
$$
(\log\mu)^2-(\log\xi_l)^2>\log\mu(\log\xi_l-\log\xi_1)
$$
where $\mu^{-1}$ is the absolute value of the stable eigenvalue, $\xi_1$ is the smallest absolute value of the eigenvalues which are greater than 1 and $\xi_l$ is the largest absolute value of the eigenvalues of $L$.

Then there exists a $C^1$ neighborhood $\cU$ of $L$ in $\textup{Diff}^r(\T^d)$, $r\ge 3$, and a $C^r$-dense $C^1$-open subset $\cV\subset\cU$ such that if $f_1,f_2\in\cV$ have matching Jacobian periodic data then $f_1$ and $f_2$ are $C^{1+\eps}$ conjugate for some $\eps>0$.
\end{theorem}
\begin{remark} The irreducibility of $L$ is, in fact, automatic from the codimension 1 assumption.
\end{remark}
\begin{remark} We remark that the new methods are employed to obtain smoothness of the conjugacy along the unstable foliation, while smoothness along the 1-dimensional stable foliation is standard~\cite{dlL}. The assumption that the stable foliation is 1-dimensional can be replaced with some more complicated assumption where older methods apply. For example, assumption like those in~\cite{GG, G} on the stable subbundle would allow for a higher dimensional stable foliation with real spectrum.
\end{remark}
The bunching assumption is a technical assumption which guarantees sufficient regularity of the Anosov splitting (in particular, both stable and unstable distributions are at least $C^1$) and $C^1$ regularity of periodic cycle functionals. Note that the bunching assumption is satisfied by automorphisms which are sufficiently close to conformal in the unstable subbundle. The set $\cV$ consists of very non-algebraic Anosov diffeomorphisms and will be explicitly described in the course of the proof. Roughly speaking the very non-algebraic condition says that Jacobians restricted along various invariant subbundles are not cohomologous to constants. Then, using these Jacobians we can obtain a supply of non-trivial matching functions. This set $\cV$ will not contain any volume preserving diffeomorphisms, however we have the following addendum.

\begin{add}
\label{add}
In the setting of the above theorem there also exists a $C^1$ neighborhood $\cU'$ of $L$ in the space of volume preserving diffeomorphisms $\textup{Diff}^r_{vol}(\T^d)$, $r\ge 3$, and a $C^r$-dense $C^1$-open subset $\cV'\subset\cU'$ such that if $f_1,f_2\in\cV'$ have matching Jacobian periodic data then $f_1$ and $f_2$ are $C^{1+\eps}$ conjugate for some $\eps>0$.
\end{add}

\subsection{Organization} The next section is devoted to background material. We begin by setting up the notation which will be used consistently throughout the paper and by recalling well-known results on regularity of invariant foliations for Anosov diffeomorophisms. Then we define periodic cycle functionals which are the primary technical tool of this paper. We explain that periodic cycle functionals provide a complete collection of obstructions to solving the cohomological equation over an Anosov diffeomorphism. Then we establish a technical lemma on $C^1$ regularity of periodic cycle functionals. The method of proof is standard, but this lemma is not present in the literature, so we included the proof. Then we also discuss the relation between periodic cycle functionals and Jacobians of stable holonomy maps. In the last subsection we recall a result on non-stationary linearization which we will need for bootstrapping regularity of the conjugacy.

Section~3 contains all the proofs for 3-dimensional Anosov diffeomorphisms with a fixed point which has complex eigenvalues in the unstable subspace. This is the simplest situation where our method yields new results. The proof is done in two steps. First $C^1$ regularity of the conjugacy along the unstable foliation is derived from matching periodic cycle functionals. Then the second step is to bootstrap regularity along unstable foliation using non-stationary linearization theory.


In Section~4 we prove what we call the General Matching Theorem which is a technical statement resulting from careful study of local spaces of matching functions. Specifically we obtain certain invariant sub-foliations of the unstable foliation associated to the local spaces of matching function and obtain a number of properties of this sub-foliation. 

Then in Section~5 we apply the General Matching Theorem to Anosov diffeomorphism on the 3-torus with real spectrum at periodic points. We are able to (partially) improve the previous rigidity result~\cite{GG} by bootstrapping regularity of the conjugacy. Besides the General Matching Theorem, the other technical ingredient of this proof is that using appropriate bunching we establish some extra regularity of periodic cycle functionals. Namely, we show that they are $C^{(5+\sqrt{17})/(3+\sqrt{17})}$ or better. This gives the same amount of extra regularity for the conjugacy which is then sufficient to proceed with non-stationary linearization bootstrap argument.

Finally in the last Section~6 we (partially) generalize the 3-dimensional results to Anosov diffeomorphisms on higher dimensional tori, especially for codimension one Anosov diffeomorphisms. This generalization uses the sub-foliation of the unstable foliation coming from the General Matching Theorem. Then the bulk of the proof consists of analyzing all possible cases for the position of this sub-foliation relative to the dominated splitting in the unstable subbundle. Unless, the conjugacy is $C^1$ we show that presence of such a sub-foliation gives some extra rigidity which is not consistent with our assumptions on Anosov diffeomorphisms.

We would like to thank the referee for many useful remarks which improved the exposition.

\section{Preliminaries}
We will denote by $W^s$ and $W^u$ the stable and unstable foliations which are tangent to the stable distribution $E^s$ and the unstable distribution $E^u$ of an Anosov diffeomorphism $f$, respectively. When it is necessary to indicate the Anosov diffeomorphism which is being considered we will write $E^s_f$, $W^s_f$, etc. By $W^s_{loc}(x)$ and $W^u_{loc}(x)$ we will denote local invariant manifolds centered at $x$ whose size is given by the local product structure constant. 

\subsection{Regularity of the stable foliation}
\label{sec_foliations_regularity}
We recall the definition of the stable bunching parameter $b^s(f)$ which controls regularity of the stable foliation, which defined in terms of exponential rates. Namely, for an Anosov diffeomorphism $f$ there exist constants $\mu_+>\mu_->1$ and $\lambda_+>\lambda_->1$ and $C>0$ such that 
\begin{multline*}
\frac1C\mu_+^{-n}\|v^s\|\le \|Df^n(v^s)\|\le C\mu_-^{-n}\|v^s\|\,\,\,\,\,\,\mbox{and}\\
 \frac1C\lambda_-^n\|v^u\|\le \|Df^n(v^u)\|\le C\lambda_+^n\|v^u\| \,\,\,\,\,\,\,\,\,\,
\end{multline*}
for all $n\ge 0$ and all $v^s\in E^s$, $v^u\in E^u$. Then the {\it stable bunching parameter} is given by
$$
b^s(f)=\frac{\log\lambda_-}{\log\lambda_+}+\frac{\log\mu_-}{\log\lambda_+}
$$
If $b^s(f)$ is not an integer (which we can always assume) then the stable foliation $W^s$ and the stable distribution $E^s$ are $C^{b^s(f)}$ regular~~\cite{HPS, Hass}. In particular, the stable holonomy maps are $C^{b^s(f)}$. (In fact, a better point-wise version of this result holds~\cite{Hass}.) Symmetrically, the unstable foliation $W^u$ and the unstable distribution $E^u$ are $C^{b^u(f)}$ regular, where the {\it unstable bunching parameter} is given by
$$
b^u(f)=\frac{\log\mu_-}{\log\mu_+}+\frac{\log\lambda_-}{\log\mu_+}
$$

Note that if the Anosov automorphism $L\colon\T^3\to\T^3$ has one dimensional stable subbundle corresponding to an eigenvalue $\mu^{-1}$, $|\mu|>1$, and a pair of complex conjugate eigenvalues $\lambda,\bar\lambda$, then for small perturbations $f$ we have $\mu_-^{-1}\ge|\mu|^{-1}$ or $\mu_-\le |\mu|=|\lambda|^2$. Also we have $\lambda_-\le|\lambda|\le\lambda_+$. Hence
$$\frac{\log\lambda_-}{\log\lambda_+}\le 1\,\,\,\,\,\,\mbox{and}\,\,\,\,\, \frac{\log\mu_-}{\log\lambda_+}\le 2
$$
For sufficiently small perturbations $f$ the above ratio will be close to 1 and 2, respectively, and, hence, the stable foliation is $C^{3-\eps}$, where $\eps>0$ can be taken arbitrarily small by controlling the size of the perturbation. In fact, we will only need $C^{2+\eps}$ regularity for $W^s$. Calculating $b^u(f)$ in this setting gives $C^{\frac32-\eps}$ regularity of $W^u$.

\subsection{Cohomological equation over Anosov diffeomorphisms and periodic cycle functionals}
\label{sec_pcf}

Here we recall an alternative approach to solving the cohomological equation $\phi=u-u\circ f+const$ over and Anosov diffeomorphisms $f$. This approach is due to Katok and Kononenko who introduced it to study the cohomological equation over partially hyperbolic diffeomorphisms~\cite{KK}. For Anosov diffeomorphisms this approach is much easier because local accessibility property always holds due to absence of the center direction, however we will need to slightly refine the argument in order to rely on a sub-collection consisting of null-homologous periodic cycle functionals only.

A piecewise smooth path $\gamma\colon[0,1]\to M$ is called a {\it $us$-adapted path} if each smooth leg is entirely contained in a stable or unstable leaf of $f$. If $\gamma(1)=\gamma(0)$ then we say that $\gamma$ is a {\it $us$-adapted loop}.

Given a H\"older continuous function $\phi\colon M\to \R$ the {\it chain functionals} are defined in the following way. If $\gamma$ lies entirely in a stable leaf then let
$$
PCF_\gamma(\phi)=\sum_{n\ge 0}\phi(f^n(\gamma(0)))-\phi(f^n(\gamma(1)))
$$
If $\gamma$ lies entirely in a stable leaf then let
$$
PCF_\gamma(\phi)=\sum_{n<0 }\phi(f^n(\gamma(1)))-\phi(f^n(\gamma(0)))
$$
Given a $us$-adapted path $\gamma=\gamma_1*\gamma_2*\ldots *\gamma_m$, with each leg $\gamma_i$ entirely contained in a stable or an unstable leaf let
$$
PCF_\gamma(\phi)=\sum_{i=1}^m PCF_{\gamma_i}(\phi)
$$
Note that the value $PCF_\gamma(\phi)$ only depends on the sequence of endpoints of $\gamma_i$. If $\gamma$ is a $us$-adapted loop then $PCF_\gamma(\phi)$ is called the {\it periodic cycle functional}.

If $\phi=u-u\circ f+const$ then, by a direct calculation
$PCF_\gamma(\phi)=u(\gamma(0))-u(\gamma(1))$. Hence values of periodic cycle functionals on $us$-adapted loops provide obstructions to solving the cohomological equation. It turns out that vanishing of these obstructions is a sufficient condition for existence of a solution.

\begin{proposition}[Katok-Kononenko]
\label{prop_KK}
If $f\colon M\to M$ is an Anosov diffeomorphism and $\phi$ is a H\"older continuous function such that $PCF_\gamma(\phi)=0$ for every $us$-adapted loop $\gamma$, then $\phi$ is  an almost coboundary; that is, there exists a constant $c$ and a H\"older continuous function $u$ such that $\phi=u-u\circ f +c$.
\end{proposition}

\begin{proof} Assume that $f$ has a fixed point $x_0$, $f(x_0)=x_0$ and let $c=\phi(x_0)$. Given a point $x\in M$ consider a $us$-adapted path $\gamma$ connecting $x_0$ to $x$ and let
$u(x)=PCF_\gamma(\phi)$. Note that $u(x)$ does not depend on choice of $\gamma$ because a different choice would adjust the value of $u(x)$ by a value of periodic cycle function on a loop, which is zero by our assumption.

By a direct calculation we have
$$
u(x)-u(f(x))=PCF_\gamma(\phi)-PCF_{f\circ \gamma}(\phi)=\phi(x_0)-\phi(x)
$$
It is standard to check H\"older continuity of $u$ by checking that restrictions to stable and unstable leaves are H\"older continuous.

In general, it in not known whether every Anosov diffeomorphism has a fixed point. However, every Anosov diffeomorphism has a periodic point $f^p(x_0)=x_0$. Then the general case can be reduced the case when there exists a fixed point using the following lemmas.~\footnote{We are grateful to the referee for pointing out the subtlety of this reduction}

Given a function $\phi$, let $S_p\phi=\sum_{k=0}^{p-1}\phi\circ f^k$. Also, for clarity, we will include the dynamics in the notation for $PCF$.
\begin{lemma} We have
	$PCF_{\gamma}(\phi,f)=PCF_{\gamma}(S_p\phi,f^p)$ for every su-loop $\gamma$.
\end{lemma}	

\begin{lemma}
	If $S_p\phi$ is cohomologous to a constant over $f^p$ then $\phi$ is cohomologous to a constant over $f$. 
\end{lemma}	

Lemma~2.2 follows from a direct calculation. Indeed, if $\gamma_i$ is, say, a stable leg of the loop $\gamma$ then
\begin{multline*}
PCF_{\gamma_i}(\phi,f)=\sum_{n\ge 0}\phi(f^n(\gamma_i(0)))-\phi(f^n(\gamma_i(1)))\\
=\sum_{j\ge 0}S_p\phi(f^{jp}(\gamma_i(0)))-S_p\phi(f^{jp}(\gamma_i(1)))=PCF_{\gamma_i}(S_p\phi, f^p)
\end{multline*}
	\begin{proof}[Proof of Lemma~2.3]
		We can assume that $p\geq 2$. Let $u$ and $c\in\R$ be such that $S_p\phi=u-u\circ f^p+c$. Define $$w=\frac{\sum_{k=1}^{p-1}S_k\phi+S_pu}{p},$$ then \begin{eqnarray*}
			p(w-w\circ f)&=&\left(\sum_{k=1}^{p-1}S_k\phi-\sum_{k=1}^{p-1}S_k\phi\circ f\right)+S_pu-S_pu\circ f\\
			&=&\left(\sum_{k=1}^{p-1}S_k\phi-S_k\phi\circ f\right)+u-u\circ f^p\\
			&=&\left(\sum_{k=1}^{p-1}\phi-\phi\circ f^k\right)+u-u\circ f^p\\
			&=&p\phi-S_p\phi+u-u\circ f^p=p\phi-c
		\end{eqnarray*}
		hence $\phi$ is cohomologous to $c/p$ over $f$. 
\end{proof}

Now, if $f$ doesn't have a fixed point but a periodic point $f^p(x_0)=x_0$, $p\ge 2$, then from Lemma~2.2 we have that $PCF$ for $f^p$ of $S_p\phi$ vanish. Since $f^p$ has a fixed point $x_0$, we conclude that $S_p\phi$ is a coboundary over $f^p$ as explained at the beginning of the proof. Then, we obtain that $\phi$ is a coboundary over $f$ by applying Lemma~2.3. 
\end{proof}

We will need a version of the above proposition which is concerned with null-homologous $us$-adapted loops, \ie loops whose homology class vanishes in $H_1(M,\Z)$. The next proposition can be easily derived from the abelian Livshits Theorem for Anosov flows given in~\cite[Theorem 3.3]{GRH2}, however, for Anosov diffeomorphisms the proof is more direct and we include it here.

\begin{proposition} Assume that $f\colon M\to M$ is an Anosov diffeomorphism such that $f^k_*n\neq n$ for all non-zero $n\in H_1(M,\Z)$ and all $k\ge 1$. Assume that $\phi$ is a H\"older continuous function such that $PCF_\gamma(\phi)=0$ for every null-homologous $us$-adapted loop $\gamma$. Then $\phi$ is an almost coboundary.
\label{prop_pcf}
\end{proposition}

\begin{proof}
Let $\tilde M$ be the universal abelian cover of $M$, that is, the cover which corresponds to the commutator subgroup $[\pi_1M,\pi_1M]$; its group of Deck transformations is given by $H_1(M,\Z)$. Assume that $x_0$ is a fixed point for $f$ (otherwise we can use a periodic point to show that $\phi$ is an almost coboundary over $f^k$ for some $k$, then the same trick as in the proof Proposition~\ref{prop_KK} yields that $\phi$ is also an almost coboundary over $f$). And let $\tilde x_0$ be a fixed point for a lift $\tilde f\colon\tilde M\to\tilde M$ of $f$. Also let $\tilde\phi$ be the lift of $\phi$.

Note that homologically trivial $us$-adapted loops on $M$ are precisely those loops which correspond to elements in the commutator subgroup of $\pi_1(M)$ and, hence, lift to loops on $\tilde M$. Therefore, the preceding proof can be repeated verbatim on $\tilde M$. Namely, if $c=\tilde\phi(\tilde x_0)$ and $\tilde u(x)=PCF_\gamma(\tilde \phi)$, where $\gamma$ connects $x_0$ to $x$, then we have a solution to the cohomological equation on $\tilde M$:
$$
\tilde \phi=\tilde u-\tilde u\circ\tilde f+c
$$

Let $w\colon H_1(M,\Z)\to\R$ be given by $w(T)=\tilde u(T(\tilde x_0))$. Then, because periodic cycle functionals are invariant under Deck transformations: $PCF_{T\circ \gamma}(\tilde\phi)=PCF_\gamma(\tilde\phi)$, $T\in H_1(M,\Z)$, one can easily verify that $w$ is a homomorphism. And also similarly, $\tilde u\circ T-\tilde u= w(T)$.

Now for any $T\in H_1(M,\Z)$ we have
$$
\tilde\phi\circ T=\tilde u\circ T-\tilde u\circ\tilde f\circ T+c=\tilde u +w(T)-\tilde u\circ\tilde f-w(f_*(T))+c=\tilde\phi+w(T)-w(f_*(T))
$$
Hence, because $\tilde\phi\circ T=\tilde\phi$, we obtain that $w(T)=w(f_*(T))$ or $w((id-f_*)T)=0$ for all $T$. By assumption $id-f_*$ has trivial kernel, hence, we conclude that $w\equiv 0$. This means that $\tilde u$ is also equivariant under the Deck group and, thus, descends to a function $u\colon M\to \R$ and gives a solution to the cohomological equation on $M$: $\phi=u-u\circ f+c$.
\end{proof}

\subsection{Regularity of simple periodic cycle functionals}
\label{sec_regularity}

The simplest $PCF$ is given by a quadruple of points. Here we show that under assumptions of Theorem~\ref{thm_tech} such $PCF$ is $C^1$ along unstable leaves.

Let $a\in W^s(b)$. Then there is a canonical holonomy map $Hol_{a,b}\colon W_{loc}^u(a)\to W_{loc}^u(b)$ which takes $a$ to $b$ and which is given by sliding along stable leaves. If stable and unstable foliations have global product structure (as is the case for Anosov diffeomorphisms on tori), then the holonomy map can be continuously extended to a continuous map $Hol_{a,b}\colon W_{}^u(a)\to W_{}^u(b)$ in a unique way.

Let $\gamma(a,b,x)$ be a $us$-adapted loop with four legs connecting $a$ to $b$, $b$ to $Hol_{a,b}(x)$, to $x$ and then back to $a$. Given a potential $\phi$ define $\rho_{a,b}^\phi\colon W^u(a)\to \R$ via the periodic cycle functional of $\gamma(a,b,x)$
$$
\rho_{a,b}^\phi(x)=PCF_{\gamma(a,b,x)}(\phi)
$$
We will call such PCF a {\it simple PCF.}
\begin{lemma}
\label{lemma1}
 Let $L\colon\T^3\to\T^3$ be an Anosov automorphism with a pair of complex conjugate eigenvalues. Fix a number $\kappa>\frac12$ and assume let $\phi\colon \T^3\to \R$ is $C^{1+\kappa}$. Let $f$ be a sufficiently small perturbation of  the automorphism $L$ such that $\lambda_+<\mu_-^\kappa$. Then $\rho_{a,b}^\phi\colon W^u(a)\to \R$ is $C^1$ regular.
\end{lemma}

\begin{proof}
Recall that by definition
\begin{multline*}
\rho_{a,b}^\phi(x)=\sum_{n\ge 0} \phi(f^n(b))-\phi(f^n(a))+\sum_{n\ge 0}\phi(f^n(x))-\phi(f^n(Hol_{a,b}(x)))\\
+\sum_{n<0}\phi(f^n(x))-\phi(f^n(a))+\sum_{n<0}\phi(f^n(b))-\phi(f^n(Hol_{a,b}(x)))
\end{multline*}

Note that the first series term is just a constant. The third series term can be easily seen to be $C^1$ along $W^u$ by calculating the formal derivative and observing that the resulting series converge uniformly and hence, by Weierstrass M-test, give a bona fide derivative of the series. The last series term is of the same nature as the third one, but precomposed with the holonomy map. Since $Hol_{a,b}$ is $C^1$, we conclude that the last term is also $C^1$ along $W^u$. Hence, it remains to analyze the second series term.

We denote by $D_uf$ the restriction of derivative $Df$ to $E^u$ and calculate the formal derivative of the second series term:
$$
\sum_{n\ge 0}D_u\phi(f^n(x))D_uf^n(x)-D_u\phi(f^n(Hol_{a,b}(x)))D_uf^n(Hol_{a,b}(x))D\,Hol_{a,b}(x)
$$
We proceed with an estimate using the triangle inequality by splitting the above series into a sum of two series. Note that the points $f^n(Hol_{a,b}(x))$ and $f^n(x)$ are close and we can identify unstable subspaces at these points using the global coordinates on $\T^3$. In this way the compositions of differentials which appear below make sense.
\begin{multline*}
\sum_{n\ge 0}D_u\phi(f^n(x))D_uf^n(x)-D_u\phi(f^n(Hol_{a,b}(x)))D_uf^n(Hol_{a,b}(x))D\,Hol_{a,b}(x)\\
=\sum_{n\ge0}\big(D_u\phi(f^n(x))-D_u\phi(f^n(Hol_{a,b}(x)))\big)D_uf^n(x)\\
+\sum_{n\ge0}D_u\phi(f^n(Hol_{a,b}(x)))\big(D_uf^n(x)-D_uf^n(Hol_{a,b}(x))D\,Hol_{a,b}(x)\big)
\end{multline*}
We will see that both series above converge uniformly. This then implies that $\rho_{a,b}^\phi$ is indeed differentiable with a continuous derivative given by the above series. 

For estimating the first series we use the fact  that $D^u\phi$ is H\"older with exponent $\kappa>1/2$. (Indeed, recall that $E^u$ is $C^1$ and $\phi$ is $C^{1+\kappa}$.)
\begin{multline*}
\|\big(D_u\phi(f^n(x))-D_u\phi(f^n(Hol_{a,b}(x)))\big)D_uf^n(x)\|\\
\le C\, dist(f^n(x), f^n(Hol_{a,b}(x)))^\kappa\lambda_+^n\le C \mu_-^{-n\kappa}\lambda_+^n
\end{multline*}
Hence, because $\mu_-^{-\kappa}\lambda_+<1$, the series converge uniformly.

To handle the second series note that $D_u\phi$ is uniformly bounded and hence, we need to estimate $D_uf^n(x)-D_uf^n(Hol_{a,b}(x))D\, Hol_{a,b}(x)$. To do that notice that
$f^n\circ Hol_{a,b}=Hol_{f^n(a),f^n(b)}\circ f^n$. Hence
\begin{multline*}
\|D_uf^n(x)-D_uf^n(Hol_{a,b}(x))D\, Hol_{a,b}(x)\|\\
=\|D_uf^n(x)-D\, Hol_{f^n(a),f^n(b)}(f^n(x)) D_uf^n(x)\| \\
\le
\|Id-D\, Hol_{f^n(a),f^n(b)}(f^n(x))\|\cdot\|D_u f^n(x)\|\le C \mu_-^{-n}\lambda_+^n
\end{multline*}
where for the bound 
$$\|Id-D\, Hol_{f^n(a),f^n(b)}(f^n(x))\|\le C dist(f^n(x), f^n(Hol_{a,b}(x)))\le C\mu_-^{-n}$$
 we used the fact $W^s$ is a $C^2$ foliation and, hence, $D\, Hol_{z,y}$ is uniformly Lipschitz in $y\in W_{loc}^s(z)$, $z\in \T^3$, and $D\, Hol_{z,z}=Id$.
Therefore, because $\lambda_+<\mu_-$, the second series also converge uniformly.

We remark that to have $C^2$ stable foliation we strongly rely on the fact that $f$ is close to $L$ and, hence, is close to conformal along the unstable subbundle as discussed in detail in Section~2.1.
\end{proof}

\subsection{Relation between the stable holonomy and simple periodic cycle functionals}
\label{sec_holonomy}
Consider a quadruple of points $a$, $b\in W^s(a)$, $x\in W^u(a)$ and $Hol_{a,b}(x)$ as in the preceding section. The Jacobian of $Hol_{a,b}$ can be calculated using the relationship $f^n\circ Hol_{a,b}=Hol_{f^n(a),f^n(b)}\circ f^n$. Indeed, taking Jacobians of both sides yields
$$
J\, Hol_{a,b}(x)=\frac{J^uf^n(x) J\,Hol_{f^n(a),f^n(b)}(f^n(x))}{J^uf^n(Hol_{a,b}(x))}
$$
Recall that $J\, Hol_{f^n(a),f^n(b)}\to 1$ as $n\to +\infty$ because holonomy is uniformly $C^1$ and $Hol_{z,z}=Id$, $z\in\T^3$. Hence, by taking logarithms and passing to the limit as $n\to+\infty$ we obtain the following expression for the Jacobian of the holonomy
$$
\log J \,Hol_{a,b}(x)=\sum_{n\ge 0} \log J^uf(f^n(x))-\log J^uf(f^n(Hol_{a,b}(x)))
$$
This formula  gives the relationship between the Jacobian of the holonomy and the simple periodic cycle functional. Namely, if $\phi=\log J^uf$ then we have
\begin{multline*}
\rho_{a,b}^\phi(x)=\log J\, Hol_{a,b}(x)-\log J\, Hol_{a,b}(a)\\
+\sum_{n<0}\phi(f^n(x))-\phi(f^n(a))+\sum_{n<0}\phi(f^n(b))-\phi(f^n(Hol_{a,b}(x)))
\end{multline*}
\begin{remark} The formula above becomes much nicer if one considers the Jacobian of holonomy relative to the conditional measures of the SRB measure of $f$. Recall that the density of such conditional measure on $W^u(a)$ normalized to be equal to 1 at $a$, is given by
$$
\theta_a(x)=\prod_{n<0}\frac {J^uf(f^n(a))}{J^uf(f^n(x))}
$$
Then the Jacobian of holonomy relative to the SRB conditional measures on $W^u(a)$ and $W^u(b)$ is given by
$$
J^{SRB} Hol_{a,b}(x)=J \,Hol_{a,b}(x)\frac{\theta_b(Hol_{a,b}(x))}{\theta_a(x)}
$$
Taking logariths and using the formula for $\rho_{a,b}^\phi$ we have
$$
\log J^{SRB} Hol_{a,b}(x)=\rho_{a,b}^\phi(x)+\log J^{SRB} Hol_{a,b}(a)
$$
Hence, up to an additive constant simple PCF is the same as logarithmic Jacobian of the holonomy relative to the conditionals of the  SRB measure.
\end{remark}
Note that for Lemma~\ref{lemma1} to apply when $\phi=\log J^uf$, we must have that $\phi$ is  $C^{1+\kappa}$, which means that $f$ has to be $C^{2+\kappa}$ regular. However we will also need to have $C^1$ regularity of $\rho_{a,b}^\phi$ when $f$ is merely $C^2$. The above formula for $\rho_{a,b}^\phi$ in terms of Jacobian of holonomy allows to do that.
\begin{lemma}
\label{lemma2}
 Let $L\colon\T^3\to\T^3$ be an Anosov automorphism with a pair of complex conjugate eigenvalues. Let $f$ be a $C^2$ diffeomorphism which is a sufficiently $C^1$ small perturbation of the automorphism $L$ and let $\phi=\log J^uf$. Then $\rho_{a,b}^\phi\colon W^u(a)\to \R$ is $C^1$ regular.
\end{lemma}
\begin{proof} The preceding formula expresses $\rho_{a,b}^\phi$ as a sum of four terms. Recall that according to the discussion in Section~\ref{sec_foliations_regularity} the map $Hol_{a,b}$ is $C^2$ (and this is why we need $f$ to be at least $C^2$). Hence the first term $J\, Hol_{a,b}$ is $C^1$ regular. The second term is just a constant. Then, $C^1$ regularity of the third and, similarly, the last series term, are easy to see by observing that $\phi$ is $C^1$, differentiating the series formally with respect to $x$ and observing exponential convergence of the resulting series.
\end{proof}

\subsection{Non-stationary linearization for expanding foliations} Let $f\colon M\to M$ be a $C^r$, $r\ge 2$, diffeomorphism which leaves invariant a continuous foliation $W$ with uniformly $C^r$ leaves. Assume that $W$ is an {\it expanding foliation}, that is 
$\|Df(v)\|>\|v\|$, for all non-zero $v\in E$ where $E\subset TM$ is the distribution tangent to $W$. The following proposition on non-stationary linearization is a special case of the normal form theory developed by Guysinsky and Katok~\cite{GK} and further refined by Kalinin and Sadovskaya~\cite{Sad, KS2}.
\begin{proposition}
\label{prop_normal_forms}
	Let $f$ be a $C^r$, $r\ge 2$, diffeomorphism and let $W$ be an expanding foliation as described above, $E=TW$. Assume that there exist $\nu\in[0,1]$ and $\gamma\in(0,1)$ such that
	$$
	\|(Df^n|_E)^{-1}\|^{1+\nu}\cdot \|Df^n|_E\|\le C\gamma^n
	$$
	for all $n\ge1$.
	Then for all $x\in M$ there exists $\cH_x:E(x)\to W^u(x)$ such that
	\begin{enumerate}
		\item $\cH_x$ is a $C^r$ diffeomorphism for all $x\in M$;
		\item $\cH_x(0)=x$;
		\item $D_0\cH_x=id$;
		\item $\cH_{fx}\circ D_xf=f\circ \cH_x$;
		\item $D\cH_x$ is Lipschitz along $W$;
		\item \label{unique} such family $\cH_x$, $x\in M$, is unique among linearizations satisfying the above properties; moreover, uniqueness still holds among linearizations which do not necessarily obey item 5 above, but with $\nu$-H\"older dependence of $D\cH_x$ along $W$;
		\item\label{aff} if $y\in W(x)$ then $\cH_y^{-1}\circ \cH_x:E(x)\to E(y)$ is affine;
		\item the map $x\to \cH_x$ from $M$ to $Imm^r(E(x), M)$ is H\"older, in particular, the map $\hat \cH:E\to M$, given by $\hat H(x,v)=H_x(v)$ is continuous;
 	\end{enumerate}
\end{proposition}

Such family $\{\cH_x, x\in M\}$ is called {\it non-stationary linearization/normal form or affine structure } along $W$.

\section{Proofs of results in dimension 3}
For all the proofs in this section we will assume that $L$ has one real eigenvalue of absolute value $<1$ and a pair of complex conjugate eigenvalues of absolute value $>1$. If the real eigenvalue has absolute value $>1$ then one can consider inverses and conclude the same results.
\subsection{Theorem~\ref{thm_tech} implies Theorem~\ref{thm_main} with a caveat} 
\label{sec_follows}
The caveat is that we need to make an additional assumption that $f_i$ are at least $C^{2+\kappa}$ regular. Under this assumption we explain that Theorem~\ref{thm_tech} applies in the setting of Theorem~\ref{thm_main}.

Fix a Riemannian metric on $\T^3$ and let $\phi_i=\log J^uf_i$, $i=1,2$. Because unstable Jacobian periodic data match we have $(f_1,\phi_1)\sim(f_2,\phi_2)$. In order to apply Theorem~\ref{thm_tech} we also need to check that $\phi_i\in C^{1+\kappa}(\T^3)$ with $\kappa>\frac12$. Note that this is not immediate because the unstable subbundle is merely $C^{\frac32-\eps}$. However, because the stable foliation $W^s_i$ is $C^2$ we can pick $C^2$- coordinate charts on $\T^3$ such that $W^s_i$ is ``horizontal'' with respect to these charts. Then the differential $Df_i$ has upper triangular form in these charts. Taking the determinant of $Df_i$ yields the following relation
$$
\phi_i=\log J^uf_i=\log Jf_i-\log J^sf_i
$$
We have $\log Jf_i\in C^{1+\kappa}(\T^3)$ because we have assumed that $f_i$ are $C^{2+\kappa}$. And $\log J^sf_i$ is also $C^{1+\kappa}$ because the stable subbundle $E^s_i$ is $C^2$. Therefore $\phi_i$ are indeed $C^{1+\kappa}$.

Applying Theorem~\ref{thm_tech} we obtain that $\log J^u f_i$ is cohomologous to a constant or the conjugacy $h$ is uniformly $C^r$ along the unstable foliation. In the former case we conclude that the equilibrium state for $-\log J^u f_i$ coincides with the equilibrium state for the constant, which precisely means that the SRB measure coincides with the measure of maximal entropy for $f_i$.

In the case when $h$ is uniformly $C^r$ along the unstable foliation we need to refer to classical arguments~\cite{dlL} to conclude that $h$ is $C^{r_*}$. Indeed, from matching of stable Jacobian periodic data de la Llave concludes that $h$ sends the SRB measure for $f_1^{-1}$ to the SRB measure for $f_2^{-1}$. Same is true for conditional measures of these SRB measures along the stable leaves. Further, de la Llave argues that these conditional measures have $C^{r-1}$ densities. And because the stable foliation is one dimensional we can conclude that $h$ is uniformly $C^{r}$ along stable foliation by integrating. Finally, given that $h$ is uniformly $C^r$ along both the stable and the unstable foliation, one employs the Journ\'e's Lemma~\cite{J}, to conclude that $h$ is a $C^{r_*}$ diffeomorphism. 

\subsection{Proofs of Corollaries}

Here we explain how Corollaries~\ref{cor1} and~\ref{cor2} follow from Theorem~\ref{thm_tech}.

\noindent {\it Proof of Corollary~\ref{cor1}.} Let $\phi_i=\log J^s f_i$, $i=1,2$. Then, from regularity of $E^s$ we have $\phi_i\in C^2(\T^3)$ and by the matching assumption $(f_1,\phi_1)\sim (f_2,\phi_2)$. Thus Theorem~\ref{thm_tech} applies and we have that either $h$ is $C^r$ along the unstable foliation, and we further get that $h$ is $C^{r_*}$ as explained in Section~\ref{sec_follows}, or $\phi_1$ is cohomologous to a constant. In the latter case, the equilibrium state for $\phi_1=\log J^sf_1=-\log J^u(f_1^{-1})$, which is the SRB measure for $f_1^{-1}$ coincides with the equilibrium state for a constant which is the MME.

\noindent {\it Proof of Corollary~\ref{cor2}.} Here we can apply Theorem~\ref{thm_tech} to full Jacobians $\phi_i=\log Jf_i$. Diffeomorphism $f_1$ being dissipative implies that $\phi_1$ is not cohomologous to a constant. Hence, Theorem~\ref{thm_tech} implies that $h$ is $C^r$ along unstable foliation. Now we have matching of full Jacobian and matching of the unstable Jacobian. Hence, the stable Jacobian also matches and we can coclude that $h$ is $C^{r_*}$ in the same way as before.

\subsection{Outline of the proof of Theorem~\ref{thm_tech} (and Theorem~\ref{thm_main})}\label{sec_33}
By the assumption $(f_1,\phi_1)\sim (f_2,\phi_2)$ we have that $\phi_1$ is cohomologous to $\phi_2\circ h$ over $f_1$. Because periodic cycle functionals vanish on coboundaries we have that
$$
PCF_\gamma(\phi_1)=PCF_\gamma(\phi_2\circ h)
$$
for every $us$-adapted loop $\gamma$ for $f_1$. Now we focus on simple PCFs given by four legs $\rho_{a,b}^\phi\colon W^u(a)\to\R$ as defined in the Section~\ref{sec_regularity}. The above equality of PFCs can be written in the following way
$$
\rho_{a,b}^{\phi_1}=\rho_{a,b}^{\phi_2\circ h}=\rho_{h(a), h(b)}^{\phi_2}\circ h|_{W^u_{f_1}(a)} 
$$
We call such relation {\it matching of functions} $\rho_{a,b}^{\phi_1}$ and $\rho_{h(a), h(b)}^{\phi_2}$. This relation holds for all $a\in\T^3$ and $b\in W^s(a)$.

Now the proof splits into two cases. The first case is when all simple PCFs $\rho_{a,b}^{\phi_1}$ are constant. In this case, we have, in fact, that $\rho_{a,b}^{\phi_1}\equiv 0$ because $\rho_{a,b}^{\phi_1}(a)=0$. We will deduce that such vanishing implies that all PCFs on null-homotopic $us$-adapted loops vanish. Then Proposition~\ref{prop_pcf} allows us to conclude that $\phi_1$ is an almost coboundary, which completes the proof in this case.

The second case in when $\rho_{a,b}^{\phi_1}$ is non-constant for some $a$ and $b$. 
Denote by $p$ a fixed point of $f_1$ such that $Df_1$ has a pair of (non-real) complex conjugate eigenvalues. Such a point exists in proximity of $0\in\T^3$ because we have assumed that $f_1$ is sufficiently close to $L$ in $C^1$ topology and $L$ has a pair of complex conjugate eigenvalues.
Recall that by Lemma~\ref{lemma1} $\rho_{a,b}^{\phi_1}$ is $C^1$. Using minimality of the stable foliation, we can adjust the locations of the points $a$ and $b$ on the stable manifold such that $a\in W^u_{f_1}(p)$ and $\rho_{a,b}^{\phi_1}$ is has non-zero differential at $p$. Note that dynamics produces another matching relation as follows
$$
\rho_{a,b}^{\phi_1}\circ f_1|_{W^u_{f_1}(p)}=\rho_{h(a), h(b)}^{\phi_2}\circ h|_{W^u_{f_1}(p)} \circ f_1|_{W^u_{f_1}(p)}=(\rho_{h(a), h(b)}^{\phi_2}\circ f_2|_{W^u_{f_2}(h(p))})\circ h|_{W^u_{f_1}(p)} 
$$
Hence we have another matching pair $(\rho_{a,b}^{\phi_1}\circ f_1|_{W^u_{f_1}(p)}, \rho_{h(a), h(b)}^{\phi_2}\circ f_2|_{W^u_{f_2}(h(p))})$.
The differential $D(\rho_{a, b}^{\phi_1}\circ f_1|_{W^u_{f_1}(p)})$ is also non-zero at $p$ and has a kernel which is linearly independent from the kernel of $D\rho_{a, b}^{\phi_1}$. This is because $Df_1(p)$ is an ``expanding rotation'' and doesn't have any real eigenvalues. Hence we have two independent matching relations on the neighborhood of of $p$ in $W^u(p)$, which makes it possible to apply the Inverse Function Theorem to conclude that $h|_{W^u_{f_1}(p)}$ is a $C^1$ diffeomorphism on a small neighborhood of $p$. Then we can use $C^1$ regularity of the stable holonomy to spread this regularity everywhere and conclude that $h$ is uniformly $C^1$ along the unstable foliation $W^u_{f_1}$. The last step in the proof is to use uniqueness of the normal form for the unstable foliation to bootstrap regularity along $W^u_{f_1}$ from $C^1$ to $C^r$. Then concluding that $h$ is a $C^{r_*}$ diffeomorphism was already explained at the end of Section~\ref{sec_follows}.

In the following three sections we fill in the details for the above outline.

\begin{remark} The proof of Theorem~\ref{thm_main} is exactly the same working with the specific potentials $\phi_i=\log J^u f_i$, $i=1,2$, to conclude that $h$ is $C^r$ along the unstable foliation. The only difference is that due to possible lack of regularity of $f_i$ one has to invoke Lemma~\ref{lemma2} instead of Lemma~\ref{lemma1}. 
\end{remark}

\subsection{Case I: vanishing of simple PCFs} \label{sec_case1} We begin the proof of Theorem~\ref{thm_tech} by considering the case when simple PCFs $\rho^{\phi_1}_{a,b}\equiv 0$ for all $a\in\T^3$ and $b\in W^s(a)$. We will prove, using induction, that $PCF_\gamma(\phi_1)=0$ for all null-homologous $us$-adapted loops $\gamma$, which handles this case by applying Proposition~\ref{prop_pcf} and concluding that $\phi_1$ is an almost coboundary.

Because $\gamma$ null-homologous, it lifts to a loop $\tilde\gamma$ on the universal cover $\R^3$ and we have $PCF_{\tilde\gamma}(\tilde \phi_1)=PCF_\gamma(\phi_1)$, where $\tilde \phi_1$ is the lift of $\phi_1$ and the PCF on the universal cover is defined in the same way using the lifted dynamics. The advantage of working on the universal cover is that the lifted foliations $\tilde W^s$ and $\tilde W^u$ have global product structure because they are close to the linear foliations for $L$. In particular, the space of unstable leaves is homeomorphic to $\R$ and, hence, is linearly ordered. We will denote by $\EuScript O(x)$ the $\R$-coordinate of $\tilde W^u(x)$, $x\in\R^3$ (and similarly for $\R$-coordinates of paths which are entirely contained in unstable leaves).

We will write $\tilde \gamma=\gamma_1*\gamma_2*\ldots *\gamma_{2k}$ and we can assume that the legs $\gamma_i$ are contained in unstable leaves for even $i$ and contained in stable leaves for odd $i$; indeed, if there are two consecutive legs in the same stable (or unstable) leaf we can just combine them into a single leg. We will run induction on $k$. If $k=2$ then the corresponding PCF is simple and vanishes by the assumption.

\begin{figure}
  \includegraphics[width=\linewidth]{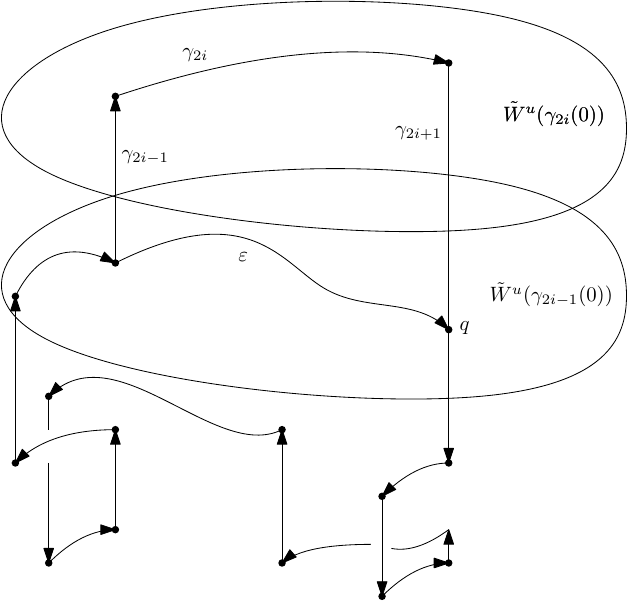}
  \caption{Induction.}
\end{figure}

Now assume vanishing for all $\tilde \gamma$ with $2k-2$ legs or less. Pick a maximal leg $\gamma_{2i}$, that is, a leg such that $\EuScript O(\gamma_{2i})\ge \EuScript O(\gamma_{2j})$ for all $j=1,\ldots k$. We can cyclically relabel the legs if needed so that $2i\neq 2k$. By maximality we have $\EuScript O(\gamma_{2i-1}(0))<\EuScript O(\gamma_{2i})$ and $\EuScript O(\gamma_{2i+1}(1))<\EuScript O(\gamma_{2i})$. For concreteness, also assume that  $\EuScript O(\gamma_{2i-1}(0))\ge \EuScript O(\gamma_{2i+1}(1))$ (the other case is symmetric). Then, by global product structure the leaf $\tilde W^u(\gamma_{2i-1}(0))$ intersects the leaf $\tilde W^s(\gamma_{2i+1}(1))$ at a unique point $q$ with $q\in \gamma_{2i+1}$. We use point $q$ to subdivide $\gamma_{2i+1}$ into two legs $\gamma_{2i+1}=\delta_1*\delta_2$. Also consider a path $\eps\colon[0,1]\to \tilde W^u(\gamma_{2i-1}(0))$ which connects $\gamma_{2i-1}(0)$ to $q$ and let $\bar\eps$ be the same path with reversed orientation which connects $q$ to $\gamma_{2i-1}(0)$. By adding the legs $\eps$ and $\bar\eps$ we can ``decompose'' $\tilde \gamma$ into two $us$-adapted loops
$$
\alpha=\gamma_{2i-1}*\gamma_{2i}*\delta_1*\bar\eps
$$
and
$$
\beta=\gamma_1*\ldots *(\gamma_{2i-2}*\eps)*\delta_2*\gamma_{2i+2}*\ldots \gamma_{2k}
$$
Note that $\alpha$ has only 4 legs and $\beta$ has $2k-2$ legs (or $2k-4$ if $\delta_2$ is a point). Hence by the induction hypothesis $PCF_\alpha(\phi_1)=PCF_\beta(\phi_1)=0$. Also recall that from the definition of periodic cycle functionals we have $PCF_{\bar\eps}(\phi_1)=-PCF_{\eps}(\phi_1)$. It follows that
$$
PCF_{\tilde\gamma}(\phi_1)=PCF_{\tilde\gamma}(\phi_1)+PCF_\eps(\phi_1)+PCF_{\bar\eps}(\phi_1)=PCF_\alpha(\phi_1)+PCF_\beta(\phi_1)=0
$$
\begin{remark} For convenience we made use of global product structure and one-dimensionality of $W^s$. However, this is not essential. A more tedious argument, which relies on local product structure only, can show that for any null-homologous $\gamma$ the corresponding $PCF$ can be written us a sum of simple $PCFs$ corresponding to loops of small diameter and, hence, vanishes.
\end{remark}

\subsection{Case II: non-constant simple PCF} \label{sec_35}Recall that the simple PCFs are $C^1$ by Lemma~\ref{lemma1}. We assume now that there exists $a\in\T^3$, $b\in W^s_{f_1}(a)$ and $x_0\in W^u_{f_1}(a)$ such that $D\rho^{\phi_1}_{a,b}(x_0)\neq 0$. Then for any $x$ in a sufficiently small neighborhood $B$ of $x_0$ we also have $D\rho^{\phi_1}_{a,b}(x)\neq 0$. 

Let $p$ be a fixed point of $f_1$ such that $Df_1|_{E^u_{f_1}(p)}$ has (non-real) complex conjugate eigenvalues. Such point exists for all $f_1$ which are sufficiently close to $L$ in $C^1$ topology.
By minimality property of $W_{f_1}^s$ we have $\T^3=\cup_{x\in B} W^s_{f_1}(x)$. Hence, we can pick $x\in B$ such that $p\in W^s_{f_1}(x)$. The local stable holonomy $Hol_{x,p}\colon W^u_{loc}(x)\to W^u_{loc}(p)$ uniquely extends to a global holonomy $Hol_{x,p}\colon W^u_{f_1}(x)\to W^u_{f_1}(p)$ and we let $c=Hol_{x,p}(a)$. The point configuration is illustrated on Fifure 2. The following relation can be verified directly from the definition of PCFs 
$$
\rho_{c,b}^{\phi_1}-\rho_{c,a}^{\phi_1}=\rho_{a,b}^{\phi_1}\circ(Hol_{x,p})^{-1}
$$
Recall that $D\rho^{\phi_1}_{a,b}(x)\neq 0$. Because $Hol_{x,p}$ is a $C^1$ diffeomorphism and $p=Hol_{x,p}(x)$, the above relation implies that either $D\rho_{c,b}^{\phi_1}(p)\neq 0$ or $D\rho_{c,a}^{\phi_1}(p)\neq 0$ (or both). These two cases are fully analogous and, for concreteness, we assume that $D\rho_{c,b}^{\phi_1}(p)\neq 0$.

\begin{figure}
  \includegraphics[width=\linewidth]{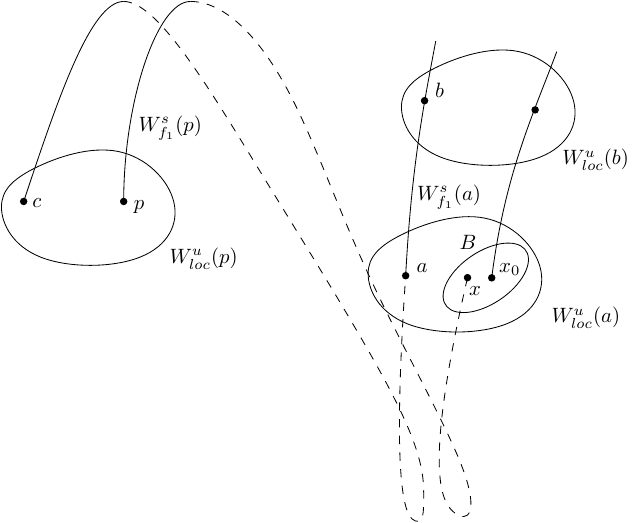}
  \caption{Point configuration.}
\end{figure}

Note that the leaf ${W^u_{f_1}(p)}$ is fixed by $f_1$. Using the conjugacy relation we now have matching pairs 
$$\rho^{\phi_1}_{c,b}=\rho^{\phi_2}_{h(c),h(b)}\circ h|_{W^u_{f_1}(p)} $$
 and 
$$\rho^{\phi_1}_{c,b}\circ f_1|_{W^u_{f_1}(p)} =(\rho^{\phi_2}_{h(c),h(b)}\circ f_2|_{W^u_{f_2}(h(p))})\circ h|_{W^u_{f_1}(p)} $$
Further the differentials $D\rho^{\phi_1}_{c,b}|_{E^u_{f_1}(p)}$ and $D(\rho^{\phi_1}_{c,b}\circ f_1|_{W^u_{f_1}(p)})|_{E^u_{f_1}(p)}$ are linearly independent because $Df_1|_{E^u_{f_1}(p)}$ does not have real eigenvalues. It follows that the map
$$
\euP_{c,b}^{\phi_1}=(\rho^{\phi_1}_{c,b}, \rho^{\phi_1}_{c,b}\circ f_1|_{W^u_{f_1}(p)})
$$
has a full-rank differential at $p$ and, hence, is a $C^1$ diffeomorphism when restricted to a sufficiently small neighborhood $\cU$ of $p$. We also define 
$$
\euP_{h(c),h(b)}^{\phi_2}=(\rho^{\phi_2}_{h(c),h(b)}, \rho^{\phi_2}_{h(c),h(b)}\circ f_2|_{W^u_{f_2}(h(p))} )$$
 and by the matching relations we have 
$$
\euP_{c,b}^{\phi_1}=\euP_{h(c),h(b)}^{\phi_2}\circ h
$$
Taking the inverse, we obtain the following formula for the restriction 
$$
h^{-1}|_{h(\cU)}=(\euP_{c,b}^{\phi_1})^{-1} \circ \euP_{h(c),h(b)}^{\phi_2}
$$
By using a symmetric argument (and passing to an even smaller neighborhood $\cU$ if needed) we also have that $h|_\cU$ is $C^1$ and, hence, $h|_\cU$ is a $C^1$ diffeomorphism.

Now let $q\in W^s_{f_1}(p)$ and let $\cU_q$ be the image of $\cU$ under the stable holonomy $Hol_{p,q}\colon W^u_{loc}(p)\to W^u_{loc}(q)$. Then, because $h$ preserves the stable foliation we have that 
$$
h|_{\cU_q}=Hol_{p,q}\circ h|_\cU\circ (Hol_{p,q})^{-1}
$$
and, hence, is also a  $C^1$ diffeomorphism. 

Finally, we will use minimality of the stable foliation to conclude that $h$ is uniformly $C^1$ along $W_{f_1}^u$. Indeed, neighborhoods $\cU_q$, $q\in W^s_{f_1}(p)$, sweep out the whole torus, so $h$ is $C^1$ along $W_{f_1}^u$. To see uniformity, note that minimality actually implies that
$$
\T^3=\bigcup_{q\in W^s(p,R)}\cU_q
$$
where $W^s(p,R)$ is an segment of radius $R$ in $W^s(p)$ relative to the intrinsic metric. Because the family of holonomies $Hol_{p,q}$, $q\in W^s(p,R)$, is uniformly $C^1$ we, indeed, can conclude that the same is true for $h|_{\cU_q}$, $q\in W^s(p,R)$, yielding uniform $C^1$ smoothness along $W^u_{f_1}$ on the whole $\T^3$.

\subsection{A bootstrap argument} 
\label{sec_37}
Denote by $\cH_x^i$, the affine structure for $f_i$ along the unstable foliation $W^u_{f_i}$, $i=1,2$. The idea for bootstrap is to use uniqueness of the normal form. Indeed, since we have that $h$ is uniformly $C^1$ along unstable foliation we can consider non-stationary linearization for $f_1$ given by ${\cH^1_x}'=(h|_{W^u_{f_1}(x)})^{-1}\circ \cH^2_x\circ Dh|_{E^u_{f_1}(x)}$. Then, if the normal form for $W_{f_1}^u$ were unique, we would have ${\cH^1_x}'=\cH^1_x$ and conclude that $h|_{W^u_{f_1}(x)}=\cH^2_x\circ Dh|_{E^u_{f_1}(x)}\circ (\cH^1_x)^{-1}$ is $C^r$. This however, does not work so easily because the uniqueness guaranteed by item~\ref{unique} of Proposition~\ref{prop_normal_forms} requires $D{\cH^1_x}'$ to be H\"older with respect to $x$ along the unstable leaves. We do not have such regularity because $h$ is merely $C^1$ along unstable leaves. However, using this idea we can still establish smoothness along the leaf with conformal fixed point and then finish using denseness of this leaf.

We will begin by bootsrtapping $h$ from $C^1$ to $C^r$ along the  leaf with conformal dynamics. We state the following proposition in somewhat more general context for the sake of future reference, in particular, in the next section.

\begin{proposition}
\label{prop_leaf}
Let $f_i\colon M\to M$ be $C^r$, $r\ge 2$, diffeomorphisms which admit 2-dimensional expanding foliations $W_i$ and satisfy assumptions of Proposition~\ref{prop_normal_forms}, $i=1,2$, $E_i=TW_i$. Assume that $f_1$ and $f_2$ are conjugate, $h\circ f_1=f_2\circ h$, and $h(W_1)=W_2$. Assume that $p$ is a fixed point for $f_1$ such that $Df_1|_{E_1(p)}$ does not have real eigenvalues. Assume that $h|_{W_1(p)}$ is differentiable at $p$. Then $h|_{W_1(p)}$ is $C^r$.
\end{proposition}
\begin{proof}
Denote by $\cH_x^i$, $x\in M$, the affine structures for $(f_i, W_i)$ given by Proposition~\ref{prop_normal_forms}. To prove the proposition we will show that $h|_{W_1(p)}=\cH^2_{h(p)}\circ Dh|_{E_1(p)}\circ (\cH^1_p)^{-1}$, which is clearly $C^r$. For that we need the following elementary lemma.
%

\begin{lemma}
	Let $H:\C\to\C$ be a continuous map of the complex plane $\C$. Assume that $D_0H$ exists. Assume that there is  $\lambda\in\C$, be such that $|\lambda|>1$ and $H(\lambda z)=\lambda H(z)$ for every $z\in\C$. Then $H(z)=D_0H(z)$ for every $z\in\C$.
\end{lemma}
\begin{proof}
	First note that $H(0)=H(\lambda 0)=\lambda H(0)$, hence, $H(0)=0$. Let $n_i\to+\infty$ such that ${\lambda^{n_i}}{|\lambda|^{-n_i}}\to\sigma$, $|\sigma|=1$. Then, since $H$ is differentiable at $0$, for every $z\in\C$ we have
	$$\frac{H(|\lambda|^{-n_i}z)}{|\lambda|^{-n_i}}\to D_0H(z)$$
	as $n_i\to+\infty$.
	 On the other hand,
	 $$
	 \frac{H(|\lambda|^{-n_i}z)}{|\lambda|^{-n_i}}=\frac{\lambda^{-n_i}}{|\lambda|^{-n_i}}H\left(\lambda^{n_i}{|\lambda|^{-n_i}z}{}\right)\to \sigma^{-1} H(\sigma z), \,\,n_i\to+\infty
	 $$
Therefore $\sigma^{-1} H(\sigma z)=D_0H(z)$ for every $z\in\C$ and hence $H(z)=\sigma D_0H(\sigma^{-1} z)$ for every $z\in\C$ which implies that $H(z)=D_0H(z)$.
\end{proof}

Because eigenvalues of $Df_1|_{E_1(p)}$ are complex, we can identify $E_1(p)$ with $\C$ so that $Df_1\colon E_1(p)\to E_1(p)$ becomes $z\mapsto \lambda z$. Note that $|\lambda|>1$. Let $H\colon E_1(p)\to E_1(p)$ be given by
$H=(\cH^1_p)^{-1}\circ(h|_{W_1(p)})^{-1}\circ\cH^2_{h(p)}\circ Dh|_{E_1(p)}$. We check the main assumption of the above lemma
\begin{multline*}
H\circ(z\mapsto \lambda z)=(\cH^1_p)^{-1}\circ(h|_{W_1(p)})^{-1}\circ\cH^2_{h(p)}\circ (Dh|_{E_1(p)}\circ Df_1|_{E_1(p)})\\
=(\cH^1_p)^{-1}\circ(h|_{W_1(p)})^{-1}\circ(\cH^2_{h(p)}\circ Df_2|_{E_2(p)})\circ Dh|_{E_1(p)}\\=(\cH^1_p)^{-1}\circ((h|_{W_1(p)})^{-1}\circ f_2|_{W_2(p)})\circ \cH^2_{h(p)}\circ Dh|_{E_1(p)}\\
=((\cH^1_p)^{-1}\circ f_1|_{W_1(p)})\circ (h|_{W_1(p)})^{-1}\circ \cH^2_{h(p)}\circ Dh|_{E_1(p)}\\
=Df_1|_{E_1(p)}\circ (\cH^1_p)^{-1}\circ(h|_{W_1(p)})^{-1}\circ\cH^2_{h(p)}\circ Dh|_{E_1(p)}=(z\mapsto \lambda z)\circ H
\end{multline*}
Also recall that $D\cH^i_x=id_{E_i(x)}$ and hence $D_0H=id_\C$. Therefore, by the lemma $H=id_\C$, which precisely means that $h|_{W_1(p)}=\cH^2_{h(p)}\circ Dh|_{E_1(p)}\circ (\cH^1_p)^{-1}$.
\end{proof}

By applying Proposition~\ref{prop_leaf} in our setting for unstable foliations, we have $h|_{W_{f_1}^u(p)}=\cH^2_{h(p)}\circ Dh|_{E^u_{f_1}(p)}\circ (\cH^1_p)^{-1}$. We would like to show a similar formula $h|_{W_{f_1}^u(x)}=\cH^2_{h(x)}\circ C(x)\circ (\cH^1_x)^{-1}$ for all $x\in\T^3$. To do that we can exploit denseness of $W^u_{f_1}(p)$. Given a point $x_0\in\T^3$ let $x_n\in W^u_{f_1}(p)$ be a sequence of points converging to $x_0$ as $n\to\infty$. For $x\in W_{f_1}^u(p)$ we have
\begin{multline*}
h|_{W^u_{f_1}(x)}=\cH^2_{h(p)}\circ Dh|_{E^u_{f_1}(p)}\circ (\cH^1_p)^{-1}\\
=\cH^2_{h(x)}\circ((\cH^2_{h(x)})^{-1}\circ\cH^2_{h(p)}\circ Dh|_{E^u_{f_1}(p)}\circ (\cH^1_p)^{-1}\circ \cH^1_x)\circ (\cH^1_x)^{-1}
\end{multline*}
and letting 
$$
C(x)=(\cH^2_{h(x)})^{-1}\circ\cH^2_{h(p)}\circ Dh|_{E^u_{f_1}(p)}\circ (\cH^1_p)^{-1}\circ \cH^1_x
$$
 we obtain $h|_{W_{f_1}^u(x)}=\cH^2_{h(x)}\circ C(x)\circ (\cH^1_x)^{-1}$ for $x\in W_{f_1}^u(p)$. By item~\ref{aff} of Proposition~\ref{prop_normal_forms} maps $C(x)\colon E^u_{f_1}(x)\to E^u_{f_2}(h(x))$ are affine. One can easily check that $C(x)(0)=0$ and hence $C(x)$, $x\in\T^3$, are, in fact, linear maps. 

Now we would like to take a limit as $x_n\to x_0$ of 
$$
h|_{W_{f_1}^u(x_n)}=\cH^2_{h(x_n)}\circ C(x_n)\circ (\cH^1_{x_n})^{-1}
$$
We left-hand-side converges to $h|_{W_{f_1}^u(x_0)}$, however, in order to be able to take the limit of the right-hand-side we also need the norm and conorm of $C(x_n)$ to be uniformly bounded. If that is the case, then from uniqueness of the limit we have that $C(x_n)$ converges to an invertible linear map $C(x_0)$, which has the same bounds on the norm and conorm, and 
$$
h|_{W_{f_1}^u(x_0)}=\cH^2_{h(x_0)}\circ C(x_0)\circ (\cH^1_{x_0})^{-1}
$$
Then from continuity property (item~8) of Proposition~\ref{prop_normal_forms} and uniform bounds on $C(x)$, $x\in\T^3$, we can conclude that $h$ is uniformly $C^r$ along unstable foliation.

Thus, it remains to prove the following lemma.
\begin{lemma}\label{agh}
	There exists a constant $C>0$ such that for every $x\in W_{f_1}^u(p)$, the linear maps $C(x):E^u_{f_1}(x)\to E^u_{f_2}(h(x))$ defined above satisfy the following bounds
	 $$\|C(x)\|\leq C,\,\, \|(C(x))^{-1}\|\leq C$$
\end{lemma}
\begin{proof}
Note that this lemma is not very obvious because the norm of $(\cH^1_p)^{-1}\circ \cH^1_x$ could explode as $x$ goes to infinity inside the leaf $W^u_{f_1}(p)$.

	We shall bound uniformly $\|C(x)\|$, $x\in\T^3$. The bound on the norm of  $C(x)^{-1}$ follows from the same argument by interchanging the roles of $f_1$ and $f_2$ and working with $h^{-1}$.
	
	Notice that $C(x)=(\cH^2_{h(x)})^{-1}\circ h|_{W^u(x)}\circ \cH^{1}_{x}.$ If there is no uniform bound on $\|C(x)\|$, then there exist sequences $x_n\in W^u(p_1)$ and $v_n\in E^u_1(x_n)$ with $\|v_n\|\to 0$ and such that $\|C(x_n)v_n\|=1$. Taking a subsequence if necessary, we have $x_{\infty}=\lim x_n\in\T^3$ and $w_{\infty}=\lim C(x_n) v_n\in E^u_2(h(x_{\infty}))$, $\|w_{\infty}\|=1$. Let $z_n=\cH^1_{x_n}(v_n),$ then $z_n\to \cH^1_{x_{\infty}}(0)=x_\infty$. So we obtain that $h(z_n)\to h(x_{\infty})$ and 
	$$
	\cH^2_{h(x_n)}(C(x_n)v_n)\to \cH^2_{h(x_{\infty})}(w_{\infty})\neq \cH^2_{h(x_{\infty})}(0)=h(x_{\infty})
	$$
	 On the other hand, 
	 $$\cH^2_{h(x_n)}(C(x_n)v_n)=h\left( \cH^{1}_{x_n}(v_n)\right)=h(z_n)\to h(x_{\infty})$$ 
	 yielding a contradiction.
\end{proof}

\begin{remark} Once Proposition~\ref{prop_leaf} is established one could argue in a more ad hoc, but quicker way, using holonomies along the stable foliation, that 
$$
h|_{W^u(x)}=Hol_{h(x_{p}),h(x)}\circ h|_{W^u(p)}\circ Hol_{x, x_{p}}
$$
where $x_{p}\in W^s(x)\cap W^u(p)$.
Hence $h|_{W^u(x)}$ is uniformly $C^2$ for every $x\in\T^3$ due to $C^2$ regularity of holonomies. After that one can use uniqueness of normal forms as outlined at the beginning of this section to further bootstrap to $C^r$. 
\end{remark}

\section{The General Matching Theorem}

\label{sec_6}
We state a general matching theorem, which is a higher dimensional version of Theorem~\ref{thm_tech}. Recall the notation for hyperbolicity rates and the defintions of the stable and unstable bunching parameters from Section~\ref{sec_foliations_regularity}.
\begin{theorem}
\label{thm_tech2}
Let $f_i\colon M_i\to M_i$, $i=1,2$, be topologically conjugate $C^r$, $r\ge 2$, transitive Anosov diffeomorphisms, $h\circ f_1=f_2\circ h$. Fix a number 
$\kappa\in(0,1]$. Assuming that both $f_i$ satisfy the following condition on the hyperbolicity rates and bunching parameters
$$
\mu_-^{-\min\{\kappa, \kappa b^u(f_i), b^s(f_i)-1\}}\lambda_+<1
$$
Then there exist $C^1$ regular, $Df_i$-invariant distributions $E_i$, such that
\begin{enumerate}
\item distributions $E_i$ integrate to $f_i$-invariant foliations $W_i$;
\item we have $\dim E_i\in [0,\ldots \dim E_{f_i}^u ]$, $E_i\subset E_{f_i}^u$ and $W_i$ subfoliates $W_{f_i}^u$;
\item the distribution $E_{f_i}^s\oplus E_i$ integrates to an $f_i$-invariant $C^1$ regular foliation which is subfoliated by both $W_{f_i}^s$ and $W_i$;
\item conjugacy $h$ maps $W_1$ to $W_2$;
\item the restrictions of $h$ to the unstable leaves are uniformly $C^1$ transversely to $W_1$; 
\item for any $C^{1+\kappa}$ functions $\phi_i\colon M\to\R$ 
such that $(f_1,\phi_1)\sim (f_2,\phi_2)$ the corresponding simple PCF $\rho^{\phi_i}_{a,b}\colon W_{f_i, loc}^u(a)\to\R$ vanish on $W_{i,loc}(a)\subset W_{f_i, loc}^u(a)$ for all $a\in M_i$ and $b\in W^s_{f_i}(a)$.
\end{enumerate}
\end{theorem}

\begin{remark} The condition on the rates is used to establish property~6.
Notice that it implies that $\mu_->\lambda_+$ and $b^s(f_i)>1$. Also note that $\mu_->\lambda_+$ implies asymmetry of stable and unstable Lyapunov spectra and
 $\dim E^s<\dim E^u$.
 \end{remark}

\begin{remark} Recall that Brin and Manning~\cite{Br,BM} proved that Anosov diffeomorphism lives on an infranilmanifold if at least one of the following pinching assumptions is satisfied
$$
1+\frac{\log \lambda_-}{\log \lambda_+}>\frac{\log\mu_+}{\log\mu_-},\,\,\,\mbox{or}\,\,\,\, 1+\frac{\log\mu_-}{\log\mu_+}>\frac{\log \lambda_+}{\log \lambda_-}
$$
 Our bunching assumption does not imply any of the Brin's pinching assumptions. For example take $\lambda_+=(\lambda_-)^2$, $\mu_-=(\lambda_-)^3$ and $\mu_+=(\lambda_-)^6$. Then both of Brin's conditions are violated, $b^s=2$, $b^u=2/3$ and, hence, our condition is satisfied for any $\kappa\in(1/2,1]$. Hence, at least in the current state of art, the above result is indeed on abstract transitive Anosov diffeomorphisms.
\end{remark}

We proceed with the proof of the matching theorem.

\subsection{The construction}
\label{sec_41}
We will fix a regularity constant $k\ge 1$, $k\le r$. The foliations subordinate to the unstable foliations which will construct will depend on the choice of this regularity constant. Later, for the proof of Theorem~\ref{thm_tech}, we will set $k=1$.

We will consider matching  $C^k$ functions on local unstable leaves. Namely, given a point $x\in M_1$ consider the space of pairs
$$
V^k_x=\{ (\phi_1,\phi_2): \,\phi_1\colon U(\phi_1)\to\R, x\in U(\phi_1)\subset W^u_{f_1,loc}(x),\, \phi_1=\phi_2\circ h,\, \phi_1, \phi_2\in C^k\}
$$
where $U(\phi_1)$ are open neighborhoods of $x$ in $W^u_{f_1,loc}(x)$. Also denote by $V^k_{x,i}$ the projections of $V^k$ to the $i$-th coordinate, $i=1,2$. Define 
$$
E_i(x)=\bigcap_{\phi\in V^k_{x,i}}\ker d_x\phi
$$
Denote $m_i(x)=\dim E_i(x)$ and let $m_i=\min_{x\in M_i}m_i(x)$.
Note that if $\phi_i\in V^k_{x,i}$ then the restriction of $\phi_i$ to a smaller open set which contains $x$ is also in $V^k_{x,i}$. From this observation and also using the conjugacy relation $h\circ f_1=f_2\circ h$, it is straightforward to verify the first of the following properties.

\begin{enumerate}
\item $Df_i (E_i(x))=E_i(f_i(x))$, $x\in M_i$ $i=1,2$;
\item functions $m_i\colon M_i\to\Z_+$ are upper semi-continuous on unstable leaves;
\item on the open set (with respect to the intrinsic topology of the unstable foliation) $\{x: m_i(x)=m_i\}$ where $m_i$ achieves its minimum the distribution $E_i$ is integrable to a foliation with $C^k$ charts.
\end{enumerate}

To check the second property, first notice that if $(\phi_1,\phi_2)\in V^k_x$ and $y\in U(\phi_1)$, the domain of $\phi_1$, then $(\phi_1,\phi_2)\in V^k_y$ as well. The linear space $E_i(x)$ can be written as a finite intersection of codimension 1 subspaces $\ker d_x\phi_i^j$, $\phi_i^j\in V^k_{x,i}$. Because $\phi_i^j$ are at least $C^1$ we have that $\ker d_y\phi_i^j$ depend continuously on $y$ and the intersection $\cap_j \ker d_x\phi_i^j$ have equal or bigger dimension (the dimension could be bigger if some of the kernels at $x$ coincide) as $\cap_j \ker d_y\phi_i^j$ for all $y$ which are sufficiently close to $x$. For all $y\in W^u_{f_i,loc}(x)$ which are sufficiently close to $x$ we have $E_i(y)\subset\cap_j \ker d_y\phi_i^j$, hence $\dim E_i(y)\le \dim E_i(x)$, \ie the dimension function is upper semi-continuous on unstable leaves. In particular, the set $\{x: m_i(x)=m_i\}$ is open on each unstable leaf.

Now, to check the third property take any $x$ such that $m_i(x)=m_i$. We can further refine the collection $\phi_i^j$ (if needed) so that the collection of differentials $\{d_x\phi_i^j\}_{j=1}^{d-m_i}$ is linearly independent. Here $d=\dim E^u_{f_i}$. Then,  by the above discussion (and minimality of the dimension function at $x$) we have 
$$
E_i(y)=\bigcap_{j=1}^{d-m_i} \ker d_y\phi_i^j
$$
for all $y\in U\subset W^u_{f_i,loc}(x)$. We can see now that that
$$
\Phi_{i,x}=(\phi_i^1,\phi_i^2,\ldots \phi_i^{d-m_i})
$$
is a foliation chart on an open neighborhood of $x$, $U$. Indeed, by linear independence of $\{d_y\phi_i^j\}_{j=1}^{d-m_i}$, $y\in U$, the  map $\Phi_{i,x}$ is a submersion, the level sets $\Phi_{i,x}=const$ have dimension $m_i$ and are tangent to $E_i$.

Note that the assumption $\mu_-^{-\kappa}\lambda_+<1$ implies that $b^s(f_i)>1$.
\begin{lemma}
\label{lemma51}
 If $k\le \min\{b^s(f_1),b^s(f_2)\}$ then $m_i(x)=m_i$ for all $x\in M_i$, $i=1,2$. 
\end{lemma}
\begin{proof}
Let $y\in W^s_{f_i}(x)$ and denote by $Hol_{x,y}\colon W^u_{f_i,loc}(x)\to W^u_{f_i,loc}(y)$ the stable holonomy map. By the assumption on $k$ this map is at least $C^k$ diffeomorphism.  Because the conjugacy $h$ sends $W^s_{f_1}$ to $W^s_{f_2}$, it respects the holonomy maps. Namely,  $h\circ Hol_{x,y}=Hol_{h(x),h(y)}\circ h$.

Now let $(\phi_1,\phi_2)\in V_x^k$. Then
$$
\phi_2\circ Hol^{-1}_{h(x),h(y)}\circ h=\phi_2\circ h\circ Hol^{-1}_{x,y}=\phi_1\circ Hol^{-1}_{x,y}
$$
Hence $\left (\phi_1\circ Hol^{-1}_{x,y},\phi_2\circ Hol^{-1}_{h(x),h(y)}\right )\in V_y^k$. Similarly, if $(\psi_1,\psi_2)\in V_y^k$ then $\left (\psi_1\circ Hol_{x,y},\psi_2\circ Hol_{h(x),h(y)}\right )\in V_x^k$. It immediately follows that $D\, Hol_{x,y}(E_i(x))=E_i(y)$ and, in particular, $m_i(x)=m_i(y)$. We can conclude that the set $\{x: m_i(x)=m_i\}$ is saturated by the stable leaves of $W^s_{f_i}$. Recall, that the stable foliation is minimal, hence, $m_i(x)=m_i$ for all $x$.
\end{proof}

By property~3 we conclude that $E_i$ integrates to a $C^k$ foliation $W_i\subset W^u_{f_i}$ (that is, $W_i$ is $C^k$ when restricted to an unstable leaf).

\begin{lemma} Foliations $W_i$ and $W_{f_i}^s$ integrate together to an $f_i$-invariant foliation.
\end{lemma}

\begin{proof} Let $\Phi_{i,y}\colon U_y\to \R^{d-m_i}$ be a foliation chart for $W_i$ as explained before. Then $\Phi_{i,y}\circ Hol_{x,y}$ (cf. the proof of Lemma~\ref{lemma51}) is a foliation chart for $W_i$ on a neighborhood of $x$. This implies that $Hol_{x,y}$ takes local leaves of $W_i$ in $Hol_{x,y}^{-1}(U_y)$ to local leaves of $W_i$ in $U_y$, which implies that $W_i$ and $W^s_{f_i}$ integrate together to a $C^k$ foliation (provided that  $k<b^s(f_i)$). The invariance is immediate from the invariance of $W_i$ and $W_{f_i}^s$.
\end{proof}

\begin{lemma} Foliations $W_1$ and $W_2$ have the same dimension and $h(W_1)=W_2$.
\label{lemma53}
\end{lemma}

\begin{proof}
Let $\Phi_{1,x}=(\phi_1^1,\phi_1^2,\ldots \phi_1^{d-m_1})$ be a chart for $W_1$ centered at $x$. Then we can express the local leaf through $x$ as the intersection of level sets
\begin{multline*}
W_{1,loc}(x)=\bigcap_{j=1}^{d-m_1}(\phi_1^j)^{-1}(\phi_1^j(x))=\bigcap_{j=1}^{d-m_1}h^{-1}((\phi_2^j)^{-1}(\phi_2^j(h(x))))\\
=h^{-1}\left(\bigcap_{j=1}^{d-m_1}(\phi_2^j)^{-1}(\phi_2^j(h(x)))\right)\supset h^{-1}\left(\bigcap_{\phi\in V_{x,2}^k}\phi^{-1}(\phi(h(x)))\right)\\
=h^{-1}(W_{2,loc}(h(x)))
\end{multline*}
Similarly, using a chart at $h(x)$ we also have $W_{2,loc}(h(x))\supset h(W_{1,loc}(x))$. Hence $h(W_1)=W_2$ and, by Invariance of Domain, $m_1=m_2$.
\end{proof}

\begin{lemma} The restrictions of $h$ to the unstable leaves are uniformly $C^k$, transversely to $W_1$.
\label{lemma_47}
\end{lemma}
\begin{proof}
Recall that $W_i$ is $C^k$ when restricted to unstable leaves and, by Lemma~\ref{lemma53}, locally, $h$ induces a homeomorphism $\bar h$ on the space of leaves of $W_i$. We need to prove that these local homeomorphisms $\bar h$ are uniformly $C^k$. 

Consider a foliation chart $\Phi_{1,x}=(\phi_1^1,\phi_1^2,\ldots \phi_1^{d-m_1})$ for $W_1$ centered at $x$ and chart $\Psi_{2,h(x)}=(\psi_2^1,\psi_2^2,\ldots \psi_2^{d-m_1})$  for $W_2$ centered at $h(x)$, which come from the spaces $V_{1,x}^k$ and $V_{2,x}^k$, respectively.
Then the induced homeomorphism $\bar h$ can be expressed in charts as
$$
\bar h(\phi_1^1(y),\phi_1^2(y),\ldots \phi_1^{d-m_1}(y))=(\psi_2^1(h(y)),\psi_2^2(h(y)),\ldots \psi_2^{d-m_1}(h(y)))
$$
Recall that $\psi_2^j\circ h=\psi_1^j$ for some $C^k$ functions $\psi_1^j$. Therefore,
$$
\bar h(\phi_1^1(y),\phi_1^2(y),\ldots \phi_1^{d-m_1}(y))=(\psi_1^1(y),\psi_1^2(y),\ldots \psi_1^{d-m_1}(y))
$$
which implies that $\bar h$ is $C^k$ because $\Phi_{1,x}$ is a $C^k$ submersion. Symmetric argument yields $C^k$ regularity of $\bar h^{-1}$. Finally it is easy to see uniformity (provided that $k\le \min\{b^s(f_1),b^s(f_2)\}$) of the $C^k$ regularity by using finitely many unstable plaques which are sufficiently dense and observing that $C^k$ regularity holds uniformly on all plaques related to chosen ones via short stable holonomies.
\end{proof}

\subsection{Proof of Theorem~\ref{thm_tech2}} We apply the construction described above with $k=1$. Note that the preceding lemmas yield all the conclusions of Theorem~\ref{thm_tech2} except for the last one. The proof of the last property is based on the following variant of the lemma about regularity of simple periodic cycle functionals. Recall the definition of simple PCF $\rho_{a,b} ^\phi$ from Section~\ref{sec_holonomy}. We will need the following variant of Lemma~\ref{lemma1}.
\begin{lemma} 
\label{lemma1bis}
If $\varphi\in C^{1+\kappa}(M,\R)$ and $f\colon M\to M$ is an Anosov diffeomorphism satisfying such that
$$
\mu_-^{-\min\{\kappa, \kappa b^u(f_i), b^s(f_i)-1\}}\lambda_+<1
$$
 then $\rho_{a,b}^\phi\colon W^u_{loc}(a)\to \R$ is $C^1$ regular.
\end{lemma}

Now the proof of property~6 follows easily. Recall that $\phi_i\in C^{1+\kappa}(M)$ and $(f_1,\phi_1)\sim(f_2,\phi_2)$. Therefore $(\rho_{a,b}^{\phi_1}, \rho_{h(a),h(b)}^{\phi_2})\in V_a^1$ because PCFs provide invariants for cohomology (see discussion at the beginning of Section~\ref{sec_pcf}) and by the above lemma these functions are $C^1$. Now, from construction of the foliation $W_1$ we conclude that $\rho_{a,b}^{\phi_1}$ is constant on local leaves of $W_1$. Hence $\rho_{a,b}^{\phi_1}|_{W_{1,loc}(a)}=0$ and $\rho_{h(a),h(b)}^{\phi_2}|_{W_{2,loc}(h(a))}=0$.

Thus it remains to establish Lemma~\ref{lemma1bis}. Note that $\rho_{a,b}^\phi$ is only defined on the local unstable manifold due to possible lack of global product structure.

\begin{proof}[Proof of Lemma~\ref{lemma1bis}]
The proof goes in the same way as the proof of Lemma~\ref{lemma1} by considering the formal derivative of the series and proving exponential convergence. However the estimate for the following term in the formal derivative is more delicate (the other terms of the formal derivative can be seen to converge easily)
$$
\sum_{n\ge 0}D_u\phi(f^n(x))D_uf^n(x)-D_u\phi(f^n(Hol_{a,b}(x)))D_uf^n(Hol_{a,b}(x))DHol_{a,b}(x)
$$
Recall that we identify tangent spaces at nearby points using finitely many smooth charts. As before we split the sum into two sums and use the triangle inequality.
\begin{multline*}
\sum_{n\ge 0}D_u\phi(f^n(x))D_uf^n(x)-D_u\phi(f^n(Hol_{a,b}(x)))D_uf^n(Hol_{a,b}(x))D\,Hol_{a,b}(x)\\
=\sum_{n\ge0}\big(D_u\phi(f^n(x))-D_u\phi(f^n(Hol_{a,b}(x)))\big)D_uf^n(x)\\
+\sum_{n\ge0}D_u\phi(f^n(Hol_{a,b}(x)))\big(D_uf^n(x)-D_uf^n(Hol_{a,b}(x))D\,Hol_{a,b}(x)\big)
\end{multline*}
For the terms in the first sum we have
\begin{multline*}
\|\big(D_u\phi(f^n(x))-D_u\phi(f^ n(Hol_{a,b}(x)))\big)D_uf^n(x)\|\\
\le C\, dist(E^u(f^n(x)), E^u(f^n(Hol_{a,b}(x))))^\kappa\lambda_+^n\\
\le C\, dist(f^n(x), f^n(Hol_{a,b}(x)))^{\kappa\min\{b^u,1\}}\lambda_+^n
\le C \mu_-^{-n\kappa\min\{b^u,1\}}\lambda_+^n
\end{multline*}
which converges by our assumption $\mu_-^{-\kappa\min\{b^u,1\}}\lambda_+<1$.

Now we use the fact that $D_u\phi$ is uniformly bounded and also the commutation relation between holonomy and dynamics to derive the following bound on the terms of the second series 
\begin{multline*}
\|D_u\phi(f^n(Hol_{a,b}(x)))\big(D_uf^n(x)-D_uf^n(Hol_{a,b}(x))D\,Hol_{a,b}(x)\big)\|\\
\le C\|D_uf^n(x)-D\, Hol_{f^n(a),f^n(b)} D_u f^n(x)\|
\end{multline*}
Pick a unit vector $v\in E^u(x)$, let $v_n=Df^n(v)$, $v^{Hol}=DHol_{a,b}(v)$ and $v^{Hol}_n=Df^n(v^{Hol})$. Also let $w_n=v_n/\|v_n\|$ and $w^{Hol}_n=v^{Hol}_n/\|v_n\|$. Note that after normalization we still have $w_n^{Hol}=DHol_{f^n(a),f^n(b)}(w_n)$. 

We have
\begin{multline*}
\|D_uf^n(x)-D\, Hol_{f^n(a),f^n(b)} D_u f^n(x)\|=\sup_{v\in E^u(x),\|v\|=1}\|v_n-v^{Hol}_n\|\\
\le C\lambda_+^n \sup_{w_n\in E^u(f^nx), \|w_n\|=1}\|w_n-w^{Hol}_n\|
\end{multline*}

Let $\bar W^u$ be a foliation in a neighborhood of $f^n(x)$ which is uniformly smooth (uniformly in $n$) and such that $\bar W^u(f^n(x))=W^u(f^n(x))$. Note that the angle between $\bar W^u$ and $W^u$ is $C^{b^u}$-H\"older function of the point.

Accordingly consider the pseudo stable holonomy map given by sliding along the leaves of  $\bar W^u$ which we denote by $\overline{Hol}\colon W^u_{loc}(f^n(x))\to \bar W^u(Hol_{f^n(a),f^n(b)}(f^n(x)))$ and define
$\bar w_n^{Hol}=D\overline{Hol}_{f^n(a),f^n(b)}(w_n)$. Then, because $\bar W^u$ is smooth and $W^s$ is $C^{b^s}$ we have that $D \overline{Hol}$ is $C^{\min\{b^s-1,1\}}$ H\"older along the stable leaf $W^s(f^n(x))$ and hence,
$$
\|w_n-\bar w_n^{Hol}\|\le C\, dist(f^n(x), f^n(Hol_{a,b}(x)))^{\min\{b^s-1,1\}}\le C\,\mu_-^{-n\min\{b^s-1,1\}}
$$
Vectors $w_n^{Hol}$ and $\bar w_n^{Hol}$ are based at the same point and have length very close to 1 because $w_n$ is a unit vector. Hence by H\"older continuity of $E^u$ we have
\begin{multline*}
\|w^{Hol}_n-\bar w_n^{Hol}\|\le C\measuredangle(w^{Hol}_n, \bar w^{Hol}_n)\le C dist(f^n(x), f^n(Hol_{a,b}(x)))^{\min\{b^u,1\}}\\
\le C\,\mu_-^{-n\min\{b^u,1\}}
\end{multline*}
Hence, using the triangle inequality we estimate the terms of the second series as follows
$$
\|D_uf^n(x)-D\, Hol_{f^n(a),f^n(b)} D_u f^n(x)\|\le C\lambda_+^n(\mu_-^{-n\min\{b^s-1,1\}}+\mu_-^{-n\min\{b^u,1\}})
$$
which also converges exponentially by our assumptions on the rates and bunching parameters. 
\end{proof}

\section{Proof of Theorem~\ref{thm_misha}}

We begin by calculating stable and unstable bunching parameters for $L$
$$
b^s(L)=1+\frac2\alpha>1+\frac{8}{\sqrt{17}+1}>2,\,\,\,\,\, b^u(L)=1+\frac{1}{1+\alpha}>1+\frac{4}{\sqrt{17}+2}>1
$$
The eigenvalues $\mu^{-1}<\lambda<\lambda^\alpha$ correspond to the partially hyperbolic splitting $T\T^3=E^s_L\oplus E^{wu}_L\oplus E^{uu}_L$. 
Also recall the notation for exponential bounds $\lambda_+$ and $\mu_-$ introduced in Section~\ref{sec_foliations_regularity}.
We fix a number $\eta\in(\alpha, \frac{1+\sqrt{17}}{4})$ and a small $\delta>0$ such that $(\lambda+\delta)^\alpha<(\lambda-\delta)^\eta$. We consider a sufficiently small $C^1$ neighborhood $\cU$ of $L$ such that 
\begin{enumerate}
\item 
$f\in \cU$ are Anosov;
\item the partially hyperbolic splitting persists;
\item $b^s(f)>2$;
\item $b^u(f)>1$;
\item $\lambda_+<(\lambda+\delta)^\alpha$;
\item $\lambda_->\lambda-\delta>1$;
\item $\mu_-^{-1}\lambda_+<(\lambda-\delta)^{-1}$
\end{enumerate}

We will denote the partially hyperbolic splitting for $f$ by $E^s_f\oplus E^{wu}_f\oplus E^{uu}_f$ and write $J^\sigma f$, $\sigma=s, wu, uu$, for corresponding Jacobians.

Define $\cV$ in the following way.
$$
\cV=\{f\in\cU:  \log Jf, \log J^{wu}f\,\,\,\mbox{and}\,\, \log J^{uu}f\,\,\, \mbox{are not cohomologous to constants}\}
$$ 
This is indeed an open condition, because a Jacobian not being cohomologous to a constant can always be detected from a pair of periodic points.
\begin{remark}
Alternatively we can consider a different set $\cV$.
$$
\cV=\{f\in\cU:  \log J^sf, \log J^{wu}f\,\,\,\mbox{and}\,\, \log J^{uu}f\,\,\, \mbox{are not cohomologous to constants}\}
$$ 
The proof would work in the same way.
\end{remark}

We apply Theorem~\ref{thm_tech2} to $f_1$ and $f_2$ with $\kappa=1$. However we have to apply the construction of Section~\ref{sec_41} rather than Theorem~\ref{thm_tech2}  per se. Namely, we use spaces of matching functions $V^k_x$ with regularity 
$$
k=\frac{2\eta+2}{2\eta+1}$$
 This construction yields invariant distributions $E_i\subset E_{f_i}^u$ which integrate to $C^k$ foliations $W_i$. 
The strategy of the proof now is to obtain posited regularity in the case when $\dim E_i=0$ and rule out the case $\dim E_i>0$. 

\subsection{Case I: $\dim E_i=0$}
Recall that by Lemma~\ref{lemma_47} restrictions of $h$ to the unstable leaves are uniformly $C^k$, transversely to $W_1$. In this case, it means that $h$ is uniformly $C^k$ along unstable leaves. We will now apply a bootstrap argument using normal forms to gain optimal regularity. 

Denote by $\cH_x^i$, the affine structure for $f_i$ along the unstable foliation $W^u_{f_i}$, $i=1,2$.  Also consider non-stationary linearization for $f_1$ given by ${\cH^1_x}'=(h|_{W^u_{f_1}(x)})^{-1}\circ \cH^2_x\circ Dh|_{E^u_{f_1}(x)}$. Note that ${\cH^1}'$ is $C^{k-1}$ along the unstable leaves. Now we would like to use uniqueness of normal form given by item~6 of Proposition~\ref{prop_normal_forms} for the expanding foliation and conclude that $\cH_1={\cH_1}'$. Recall that according to Proposition~\ref{prop_normal_forms} in order for uniqueness to hold we need
$$
	\|(Df_1^n|_{E_{f_1}^u})^{-1}\|^{k}\cdot \|Df^n|_{E_{f_1}^u}\|\le C\gamma^n
$$ 
for some $C>0$ and $\gamma<1$. We have
\begin{multline*}
	\|(Df_1^n|_{E_{f_1}^u})^{-1}\|^{k}\cdot \|Df^n|_{E_{f_1}^u}\|\le C(\lambda_-^{-k})^n(\lambda_+)^n\\
	\le C((\lambda-\delta)^{-k}(\lambda+\delta)^\alpha)^n
	\le C((\lambda-\delta)^{\eta-k})^n
\end{multline*}
Hence we need to have $k>\eta$ which is equivalent to
$$
2\eta^2-\eta-2<0
$$
It is easy to see what this inequality holds because $\eta<\frac{1+\sqrt{17}}{4}$.
Hence we obtain ${\cH^1_x}'=\cH^1_x$ which means that $h|_{W^u_{f_1}(x)}=\cH^2_x\circ Dh|_{E^u_{f_1}(x)}\circ (\cH^1_x)^{-1}$. Hence $h$ is as smooth as diffeomorphisms $f_i$ along the unstable foliation.

The proof in this case completes in the usual way. The matching of stable Jacobian gives smoothness of $h$ along $W^s_{f_1}$ and then we can apply Journ\'e Lemma to finish.

\subsection{Case II: $\dim E_i=1$} Recall that the integral foliation $W_1$ of the distribution $E_1$ is a $C^k$ sub-foliation of $W^u_{f_1}$ which integrates together with $W^s_{f_1}$. From invariance of $W_1$ we have that either $W_1=W^{wu}_{f_1}$ or $W_1=W^{uu}_{f_1}$. If $W_1=W^{wu}_{f_1}$ then by~\cite[Lemma~3]{G3} Lipschitz property of $W^{wu}_{f_1}$ implies that the strong unstable Jacobian $J^{uu}f_1$ is cohomologous to a constant, which contradicts to $f_1\in\cV$. If $W_1=W^{uu}_{f_1}$ then a similar result~\cite{GShi} implies that the weak unstable Jacobian $J^{wu}f_1$ is cohomologous to a constant contradicting to $f_1\in\cV$ again. (In the proof of Theorem~\ref{thm_codim1} in the next section we will explain how arguments from~\cite{G3, GShi} work in higher dimensional setting as well; see also~\cite{GS}.)

\subsection{Case III: $\dim E_i=2$}

We will need an improved version of Lemma~\ref{lemma1}.

\begin{lemma}
\label{lemma_52}
 If $f\in \cU$ and $\phi\in C^2(\T^3)$ then $\rho_{a,b}^\phi\colon W^u(a)\to \R$ is $C^k$.
\end{lemma}
We postpone the proof of this lemma until the end of the section and complete elimination of this case first. 

Let $\phi_i=\log Jf_i$. Then we have $(\phi_1,f_1)\equiv (\phi_2,f_2)$ and because PCFs are cohomology invariants (see Section~\ref{sec_33}) we have the matching relations
$$
\rho_{a,b}^{\phi_1}=\rho_{h(a), h(b)}^{\phi_2}\circ h|_{W^u_{f_1}(a)} 
$$
By Lemma~\ref{lemma_52} both $\rho_{a,b}^{\phi_1}$ and $\rho_{h(a), h(b)}$ are $C^k$. Hence $(\rho_{a,b}^{\phi_1}, \rho_{h(a), h(b)})\in V_1^k$ and item~6 of Theorem~\ref{thm_tech2} applies and yields vanishing of $\rho_{a,b}^{\phi_1}$ on $W_{1,loc}(a)=W^u_{f_1,loc}$. Now exactly the same arguments as in Section~\ref{sec_case1} apply and give vanishing of all null-homotopic PCFs for $\phi_1$, which, by Proposition~\ref{prop_pcf}, gives that $\phi_1=\log Jf_1$ is an almost coboundary contradicting to $f_1\in\cV$ again.

Thus it remains to prove Lemma~\ref{lemma_52}.

\subsection{Proof of Lemma~\ref{lemma_52}} While the setting is different (real eigenvalues rather than complex) the proof of Lemma~\ref{lemma1} with $\kappa=1$ applies without any change to give uniform $C^1$ regularity of $\rho_{a,b}^\phi$. Hence we need to gain additional regularity by establishing H\"older property of $D\rho_{a,b}^\phi$ with H\"older exponent $k-1=1/({2\eta+1})$.

Recall that $\rho_{a,b}^\phi(x)$ is a sum of four terms, the first being a constant and the third one easily seen to be smooth. The forth term has the same nature as the third one, but precomposed with $Hol_{a,b}$. The stable holonomy map is $C^{b_s(f)}$ with $b_s(f)>2$ and, hence, this term is also at least $C^2$. The second term
$$
\xi_{a,b}^\phi(x)=\sum_{n\ge 0}\phi(f^n(x))-\phi(f^n(Hol_{a,b}(x)))
$$
 is the difficult one and from the proof of Lemma~\ref{lemma1} we have that it is $C^1$ with
 $D\xi_{a,b}^\phi(x)$ given by
 $$
 \sum_{n\ge 0}D_u\phi(f^n(x))D_uf^n(x)-D_u\phi(f^n(Hol_{a,b}(x)))D_uf^n(Hol_{a,b}(x))D\,Hol_{a,b}(x)
$$
Moreover we have an exponential estimate $C_1(\mu_-^{-1}\lambda_+)^n$ on the absolute value of the $n$-th term of the series (recall that $\kappa=1$), which we can further bound using our definition of $\cU$
$$
C(\mu_-^{-1}\lambda_+)^n<C(\lambda-\delta)^{-n}
$$

Now pick two close-by points $x$ and $y\in W^u_{f_1,loc}(x)$. We will write $d$ for the metric induced by the Riemannian metric on the unstable leaves. Let $K$ be the integer which satisfies
$$
(\lambda-\delta)^{-K}<d(x,y)\le (\lambda-\delta)^{-K+1}
$$
Let $N=\lfloor(k-1)K\rfloor$ and we begin to estimate
\begin{multline*}
|D\xi_{a,b}^\phi(x)-D\xi_{a,b}^\phi(y)|\le  \sum_{n= 0}^N |D_u\phi(f^n(x))D_uf^n(x)-D_u\phi(f^n(y))D_uf^n(y)|\\
+\sum_{n= 0}^N |D_u\phi(f^n(x'))D_uf^n(x')D\,Hol_{a,b}(x)-D_u\phi(f^n(y'))D_uf^n(y')D\,Hol_{a,b}(y)|\\
+\sum_{n>N}|D_u\phi(f^n(x))D_uf^n(x)-D_u\phi(f^n(Hol_{a,b}(x)))D_uf^n(Hol_{a,b}(x))D\,Hol_{a,b}(x)|\\
+\sum_{n>N}|D_u\phi(f^n(y))D_uf^n(y)-D_u\phi(f^n(Hol_{a,b}(y)))D_uf^n(Hol_{a,b}(y))D\,Hol_{a,b}(y)|
\end{multline*}
where we write $x'$ and $y'$ for $Hol_{a,b}(x)$ and $Hol_{a,b}(y)$, respectively.
We can estimate the last two sums using the above bound on the tail of the series. Namely, they are bounded by $C_1(\lambda-\delta)^{-N}$.

We proceed to estimate the terms of the first sum using the triangle inequality and Lipschitz property of $D_u\phi$ along unstable leaves.
\begin{multline*}
|D_u\phi(f^n(x))D_uf^n(x)-D_u\phi(f^n(y))D_uf^n(y)|\\
\le |D_u\phi(f^n(x))D_uf^n(x)-D_u\phi(f^n(y))D_uf^n(x)|\\
+|D_u\phi(f^n(y))D_uf^n(x)-D_u\phi(f^n(y))D_uf^n(y)|\\
\le C\lambda_+^n|D_u\phi(f^n(x))-D_u\phi(f^n(y))|+C_2|D_uf^n(x)-D_uf^n(y)|\\
\le C_3\lambda_+^n d(f^n(x), f^n(y))+C_2|D_uf^n(x)-D_uf^n(y)|\\
\le C_4 \lambda_+^{2n} d(x,y)+ C_2C_5\lambda_+^{2n}d(x,y)\le C_6 \lambda_+^{2n}d(x,y)
\end{multline*}
where the estimate $|D_uf^n(x)-D_uf^n(y)|\le C_5\lambda_+^{2n}d(x,y)$ can be easily established by induction on $n$. Summing up to $N$ we  obtain
$$
\sum_{n= 0}^N |D_u\phi(f^n(x))D_uf^n(x)-D_u\phi(f^n(y))D_uf^n(y)|\le C_7\lambda_+^{2N}
$$

The terms of the second sum can be bounded in a similar way
\begin{multline*}
|D_u\phi(f^n(x'))D_uf^n(x')D\,Hol_{a,b}(x)-D_u\phi(f^n(y'))D_uf^n(y')D\,Hol_{a,b}(y)|\\
\le |D_u\phi(f^n(x'))D_uf^n(x')D\,Hol_{a,b}(x)-D_u\phi(f^n(y'))D_uf^n(y')D\,Hol_{a,b}(x)|\\
+|D_u\phi(f^n(y'))D_uf^n(y')D\,Hol_{a,b}(x)-D_u\phi(f^n(y'))D_uf^n(y')D\,Hol_{a,b}(y)|\\
\le C_8  |D_u\phi(f^n(x'))D_uf^n(x')-D_u\phi(f^n(y'))D_uf^n(y')|\\
+C_9\lambda_+^n|D\,Hol_{a,b}(x)-D\,Hol_{a,b}(y)|\
\end{multline*}
To estimate the first summand has exactly the same nature as the one we have estimated for the first sum, for the second summand we recall that $Hol_{a,b}$ is $C^2$ and we can use Lipschitz property of $D\,Hol_{a,b}$. Hence
\begin{multline*}
|D_u\phi(f^n(x'))D_uf^n(x')D\,Hol_{a,b}(x)-D_u\phi(f^n(y'))D_uf^n(y')D\,Hol_{a,b}(y)|\\
\le C_6\lambda_+^{2n}d(x',y')+C_{10}\lambda_+^nd(x,y)\le C_{11}\lambda_+^{2n} d(x,y)
\end{multline*}
Hence, summing up to $N$ we have exact same bound $C\lambda^{2N}$ for the second sum. Thus, putting all that together we have
\begin{multline*}
|D\xi_{a,b}^\phi(x)-D\xi_{a,b}^\phi(y)|\le  C_{12}\lambda_+^{2N} d(x,y)+C_1(\lambda-\delta)^{-N}\\
\le C_{12}(\lambda+\delta)^{2N} d(x,y)^{2-k}d(x,y)^{k-1}+C_{13} d(x,y)^{k-1}\\
\le (C_{12}(\lambda+\delta)^{2N}(\lambda-\delta)^{(2-k)(1-K)}  + C_{13})d(x,y)^{k-1}\\
\le (C_{14}(\lambda-\delta)^{2N\frac\eta\alpha}(\lambda-\delta)^{-N\frac{2-k}{k-1}}  + C_{13})d(x,y)^{k-1}
\end{multline*}
Hence, to establish the H\"older property it remains to show that $\frac{2-k}{k-1}-2\frac\eta\alpha>0$. Recall that $\frac{1}{k-1}=2\eta+1$ and that $\alpha>1$. So, indeed, we have
$$
\frac{2-k}{k-1}-2\frac\eta\alpha=2\eta-\frac{2\eta}{\alpha}>0
$$

\section{Proof of Theorem~\ref{thm_codim1}} 

\subsection{The setting} We begin by explaining how the choice of the neighborhood $\cU$ and the open dense $\cV$ is made.

Let $\xi_1<\xi_2<\ldots <\xi_l$ be the absolute values of unstable eigenvalues of $L$ and let $\mu^{-1}$ be the absolute value of the stable eigenvalue. Clearly $\mu>\xi_l$. By the definition of the stable bunching parameter we have
$$
b^s(L)=\frac{\log \xi_1+\log\mu}{\log\xi_l}>1
$$
because $\mu>\xi_l$. The bunching condition in Theorem~\ref{thm_codim1} can be rewritten as
$$
\mu^{b^s(L)-1}>\xi_l
$$
The unstable bunching parameter is given by
$$
b^u(L)=1+\frac{\log \xi_1}{\log\mu}>1
$$
By our assumption the subbundle $E_L^j$ corresponding to the eigenvalues with absolute value $\xi_j$ is either one- or two-dimesional. Then for sufficiently $C^1$-small perturbations $f$ this dominated splitting survives
$$
E_f^u=E_f^1\oplus E_f^2\oplus\ldots \oplus E_f^l
$$
Then neighborhood $\cU$ is chosen such that $f\in\cU$ are Anosov diffeomorphisms admitting the above dominated splitting for $E_f^u$, with $b^s(f)>1$ and $b^u(f)>1$, and such that $\mu^{b^s(f)-1}>\xi_l$.

Denote by $J^jf$ the Jacobian of $Df|_{E_f^j}$ and define $\cV$ in the following way.
$$
\cV=\{f\in\cU:  \log Jf, \log J^1f, \ldots \log J^lf\,\,\,\, \mbox{are not cohomologous to constants}\}
$$ 
Note that $J^jf$ not being cohomologous to a constant is indeed an open condition. Indeed, the function $\log J^jf$ is not cohomologous to a constant if and only if there exists a pair of periodic points with different sums over the smallest common period, which is an open property in $C^1$ topology. Note that if $\log Jf$ is cohomologous to a constant then $f$ is volume preserving and the constant has to be equal to 1. Also it shouldn't be very difficult to show that if $\log J^jf$ is cohomologous to a constant then this constant has to be equal to $\log J^jL$, but we won't need this.

\subsection{Outline of the proof} We apply Theorem~\ref{thm_tech2} to $f_1$ and $f_2$ with $\kappa=1$. Because $\kappa=1$, $b^u(f_i)>1$ and $\mu>\xi_l$ the assumption of Theorem~\ref{thm_tech2} boils down to
$$
\mu^{b^s(f_i)-1}>\xi_l
$$
which is satisfied according to our choice of the neighborhood $\cU$.

Theorem~\ref{thm_tech2} yields distributions $E_i\subset W_{f_i}^u$ and corresponding foliations $W_i$. If $\dim E_i=0$ then corresponding foliations are foliations by points and from conclusion of Theorem~\ref{thm_tech2} we obtain that $h$ is uniformly $C^1$ along the unstable foliation. In this case we have that $h$ matches all the intermediate distributions (because they are characterized by the speed) and one can conclude from this (see~\cite{GKS}) that $h$ is, in fact, $C^{1+\eps}$ smooth along the unstable foliation for some $\eps>0$. Matching of the stable Jacobians give smoothness along the 1-dimensional stable foliation and then Theorem~\ref{thm_codim1} follows by applying the Journ\'e Lemma.

Hence we need to rule out the case when $\dim E_i>0$. From matching of full Jacobians we have $(f_1, Jf_1)\sim (f_2, Jf_2)$. Also note that $Jf_i $ are $C^2$ functions. Hence, item~6 of Theorem~\ref{thm_tech2} applies to $Jf_i$ and we have vanishing of simple PCFs $\rho^{Jf_i}_{a,b}$ along the leaves of $W_i$. If $\dim E_i=\dim E^u_{f_i}$ then we have vanishing $\rho^{Jf_i}_{a,b}$ on local unstable leaves, which implies, by Proposition~\ref{prop_pcf}, that $Jf_i$ are cohomologous to $1$. But this means that $f_i$ are volume preserving contradicting our choice of $\cV$. 

Thus it remains to address the case when $\dim E_i\in [1, \dim E^u_{f_i}-1]$ (recall that $\dim E_1=\dim E_2$). Because $E_i$ is $Df_i$-invariant it must be a direct sum of some sub-collection of $E^j_{f_i}$.
\begin{remark} These sub-collections must be the same for $E_1$ and for $E_2$, which can seen from the fact that $h(W_1)=W_2$ and both $f_1$ and $f_2$ are $C^1$ close to $L$. 
\end{remark}
However, we will not need the above remark for the proof as we will proceed working solely with $f_1$. Namely, we will show that, in this case $f_1\notin\cV$ which yields a contradiction. We can lighten notation now by writing $f$ for $f_1$, $E$ for $E_1$, etc. We will split the argument into two sub-cases. The first case is when $E^1_f\subset E$ and the second one is when $E^1_f\nsubset E$.
Both will be treated using similar arguments relying on an idea from~\cite{G3} which allows to obtain constancy of periodic data transversely to a Lipschitz foliation, yet there are a few differences. 

\subsection{Case I: when $E^1_f\subset E$}
Denote by $j$ the smallest index such that $E_f^j\nsubset E$. Such $j$ exists because $\dim E<\dim E^u_f$. Denote by $W^{j-1,u}$ and $W^{j,u}$ the integral foliations of $E_f^1\oplus E_f^2\oplus\ldots \oplus E_f^{j-1}$ and $E_f^1\oplus E_f^2\oplus\ldots \oplus E_f^j$, respectively. From definition of $j$ it is clear that
$W^{j-1,u}=W\cap W^{j,u}$. Recall that by Theorem~\ref{thm_tech2} foliation $W$ is $C^1$ inside of the leaves of $W_f^u$. This implies that $W^{j-1,u}$ is $C^1$ inside $W^{j,u}$. Indeed this is easy to see by looking at holonomy maps. If $Hol: P_1\to P_2$ is a holonomy of $W^{j-1,u}$ between two transversals $P_1$ and $P_2$ in $W^{j,u}$, then these transversals can be embedded into transversals to $W$ inside $W^u_f$. In this way $Hol$ becomes a restriction of a holomomy map for $W$ and, hence, it is $C^1$.
\begin{lemma}
\label{lemma_lipschitz}
 If $W^{j-1,u}$ is Lipschitz inside $W^{j,u}$ then the Jacobian $J^jf$ is cohomologous to constant.
\end{lemma}

Note the lemma immediately implies that $f$ is not in $\cV$ by the definition of $\cV$, which rules out Case I. Hence it remains to prove the lemma.

\begin{proof} Recall that the subbundle $E^j_f$ is either 1-dimensional or 2-dimensional by the generic assumption on the model $L$. When $\dim E^j_f=1$ this lemma was proved in~\cite{G3}. (In the setting of~\cite{G3} the foliation $W^{j-1,u}$ was also one dimensional, but this is not affecting the argument as was remarked in~\cite{G3}.) Thus we will focus on the case when $\dim E^j_f=2$ in which case the idea is the same, but the argument has to be modified. We will be  brief when explaining the steps which are exactly the same as in~\cite{G3}.

\begin{remark} In fact, a more difficult argument work to establish rigidity of periodic data transversely to a Lipschitz foliation even when the codimension is greater than 2  (in preparation,~\cite{GS}). 
\end{remark}

For the proof of the lemma we rename foliations now in order to avoid having so many decorations: $\cT=W^{j,u}$, $W^{wu}=W^{j-1,u}$ and $E^{wu}=E_f^1\oplus E_f^2\oplus\ldots \oplus E_f^{j-1}$ (here $wu$ stand for ``weak unstable"). Also denote by $W^{su}$ the integral foliation of $E^{j,u}_f$. Then recall that we have 
\begin{enumerate}
\item
 $W^{wu}$ and $W^{su}$ are transverse sub-foliations of $\cT$;
 \item  $W^{wu}$ is conjugated to a minimal linear $L$-invariant foliation (this is standard fact for weak unstable foliations of perturbations of $L$); 
 \item $W^{wu}$ is $C^1$ in $\cT$.
 \end{enumerate}
 Using these three properties we will show that that $J^{su}$, the Jacobian of $Df|_{E^{su}}$ is constant. By Livshits theorem, it is sufficient to show that $J^{su}$ is constant on periodic points, which is what we are going to do.
 
Let $p$ be a periodic point, $f^{n_p}(p)=p$. Denote by $\lambda_p^{n_p}$ and $\mu_p^{n_p}$ the absolute values of the eigenvalues of $Df^{n_p}|_{E^{su}(p)}$. We assume that $\lambda_p\ge\mu_p$. We will show that $\lambda_p$ is a constant, which does not depend on $p$. Then an analogous argument would yield that $\mu_p$ is independent of $p$. Because $E^{su}$ is 2-dimensional we have  that Jacobian of $Df^{n_p}|_{E^{su}(p)}$  equals to $(\lambda_p\mu_p)^{n_p}$ and, hence $J^{su}$ has constant periodic data completing the proof of the lemma.

Let $\lambda_-=\inf_p\lambda_p$ and $\lambda_+=\sup_p\lambda_p$. If $\lambda_-=\lambda_+$ then we are done, so assume that $\lambda_-<\lambda_+$. Given any $\eps>0$ consider an adapted Riemannian metric  for which we have
$$
\lambda_--\eps<\frac{\|Df(v)\|}{\|v\|}<\lambda_++\eps, \,\,v\in E^{su}
$$

Now pick a periodic point $b$ with $\lambda_b<\lambda_+$ and a small $\delta>0$ such that $\lambda_b+\delta<\lambda_+$. Next we will pick a periodic point $a$ with $\lambda_a>\lambda_b+\delta$ very close to $\lambda_+$. Namely, we need that
$$
\nu=\frac{(\lambda_b+\delta)^\gamma (\lambda_++\eps)^{1-\gamma}}{\lambda_a}<1
$$
For any $\gamma>0$ we can choose an $\eps>0$ such that the above inequality holds. The choice of the value of $\gamma>0$ will be explained later.

By Hartman's Theorem~\cite{H} the contraction $f^{-n_a}\colon W^{su}_{f, loc}(a)\to W^{su}_{f, loc}(a)$ is $C^1$-linearizable. Hence there is an invariant submanifold of $W^{su}_{f, loc}(a)$ corresponding to the eigenvalue with absolute $\lambda_a$ (if the eigenvalue is complex then this submanifold is the whole $W^{su}_{f, loc}(a)$, otherwise it is 1-dimensional). If $\tilde a$ is a point on this submanifold then, using the linearization at $a$, we have,
$$
d^{su}(f^{-k}(a), f^{-k}(\tilde a))\le C \lambda_a^{-k} d^{su}(a,\tilde a),\, k\ge 0
$$
Here $d^{su}$ refers to the intrinsic metric on the leaves of $W^{su}_f$ induced by the Riemannian metric.

Consider a large disk $\cD^{wu}(D)\subset W^{wu}_f(a)$, $D>0$, given by
$$
\cD^{wu}(D)=\{x\in W^{wu}_f(a): d^{su}(x,a)<D\}
$$
Let $H_f$ be the conjugacy to the linear model $H_f\circ f=L\circ H_f$. Then we have $H_f(\cD^{wu}(D))\subset W_L^{wu}(H_f(a))$. Foliation $W_L^{wu}$ is an irrational Diophantine foliation, which implies that $H_f(\cD^{wu}(D))$ is $C\, D^{-\beta}$-dense in $\T^d$ for some $\beta>0$ and $C>0$ which is independent of $D$. Since $H_f$ is  bi-H\"older continuous, the same is true on the non-linear side --- $\cD^{wu}(D)$ is $C\, D^{-\alpha}$-dense for some $\alpha>0$ and $C>0$ which are independent of $D$.

Pick a point $c\in \cD^{wu}(D)$ which is $C\, D^{-\alpha}$ close to $b$. We now explain the choice of point $\tilde a$ on the submanifold of $W^{su}_{f, loc}(a)$ corresponding to the eigenvalue with absolute value $\lambda_a$. If $\theta$ is the H\"older exponent of subbundle $E^{su}$ then for points $y$ in a small neighborhood $B(\mathcal O(b), 2K\delta^{\frac1\theta})$ of the orbit of $b$,  we will have $\|Df|_{E^{su}(y)}\|<\lambda_b+\delta$. (Recall that the number $\delta$ was already chosen.)
Point $c$ is extremely close to $b$ for all large $D$ and we pick a point $\tilde c\in W^{su}_{f,loc}(c)$ such that it satisfies the following  properties:
\begin{enumerate}
\item $d^{su}(c,\tilde c)\ge K\delta^{\frac1\theta}$
\item $\tilde c=W^{su}_{f,loc} (c)\cap W_f^{wu}(\tilde a)$,  that is, point $\tilde c$ is related to a point point $\tilde a \in W^{su}_{f, loc}(a)$ via a $W_f^{wu}$-holonomy, where $\tilde a$ belongs to the invariant submanifold of $W^{su}_{f, loc}(a)$ corresponding to the eigenvalue with absolute value $\lambda_a$.
\end{enumerate}

We have that the pair of points $(c,\tilde c)$ is an image of the pair $(a,\tilde a)$ under a long weak-unstable holonomy inside $\cT$. We will iterate this quadruple of points $N$ times in negative time, where $N$ is chosen in such a way that $d^{wu}(f^{-N}(a), f^{-N}(c))\approx 1$. Then we have that $\log D\simeq N$, and, because $dist(c,b)\le C\, D^{-\alpha}$ we can conclude that there is a positive proportion $\gamma\in(0,1)$ of the orbit $c, f^{-1}(c), \ldots f^{-N}(c)$ such that the iterates $c, f^{-1}(c), \ldots f^{-\lfloor\gamma N\rfloor}(c)$ remain  $K\delta^{\frac1\theta}$-close to the orbit of $b$. By our choice of point $\tilde c$, the orbit of $\tilde c$ also stays $K\delta^{\frac1\theta}$-close to the orbit of $b$ and using the that $\|Df|_{E^{su}(y)}\|<\lambda_b+\delta$ we can estimate
$$
d^{su}(f^{-\lfloor\gamma N\rfloor}(c), f^{-\lfloor\gamma N\rfloor}(\tilde c))\ge (\lambda_b+\delta)^{-\lfloor\gamma N\rfloor} d^{su}(c, \tilde c)
$$
The number $\gamma$ depends on $\alpha$ and the expansion/contraction rates, but is independent of $D$ and, hence, $N$ and also is independent of $\delta$, (for all sufficiently large $D$). (The precise of the value of constant $\gamma$ can be calculated in terms of $\alpha$ and hyperbolicity rates; such calculation appears in~\cite{G3}, however we only need to know that $\gamma>0$.). 

\begin{figure}
  \includegraphics[width=\linewidth]{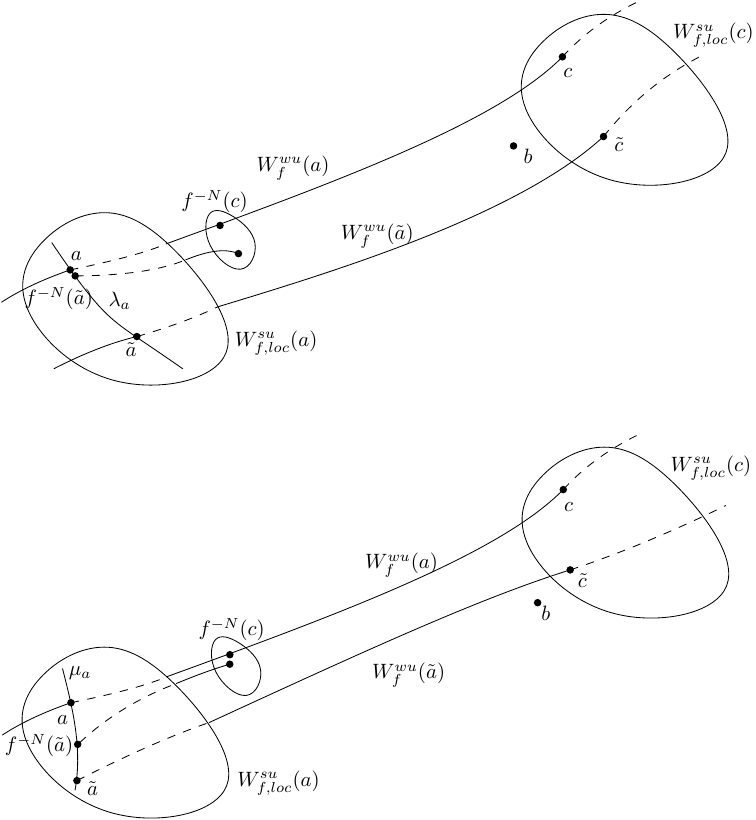}
  \caption{Point configuration.}
\end{figure}

Our goal now is to show that the ratio
$$
\frac{d^{su}(f^{-N}(a), f^{-N}(\tilde a))} {d^{su}(f^{-N}(c), f^{-N}(\tilde c))}
$$
can be arbitrarily small. To show that we need one more observation.
$$
d^{su}(a, \tilde a)\le c_0(\delta)
$$
where $c_0(\delta)$ is a constant which is independent of $D$ and $N$. Indeed,  this follows from $d^{su}(c,\tilde c)\ge K\delta^{\frac1\theta}$ and the fact that the foliation $W^{wu}_f$ is conjugate to a linear foliation via a map $H_f$ which is close to $id_{\T^d}$. 

We have 
\begin{multline*}
\frac{d^{su}(f^{-N}(a), f^{-N}(\tilde a))} {d^{su}(f^{-N}(c), f^{-N}(\tilde c))}\le \frac{C\lambda_a^{-N}d^{su}(a,\tilde a)}{(\lambda_b+\delta)^{-\lfloor\gamma N\rfloor}(\lambda_++\eps)^{\lfloor\gamma N\rfloor-N} d^{su}(c,\tilde c)}\\
\le \frac{Cc_0(\delta)}{K\delta^{1/\theta}}\left(\frac{(\lambda_b+\delta)^{\frac{\lfloor\gamma N\rfloor}{N}}(\lambda_++\eps)^{1-\frac{\lfloor\gamma N\rfloor}{N}}}{\lambda_a}\right)^N
\le  \frac{Cc_0(\delta)}{K\delta^{1/\theta}}\frac{\lambda_++\eps}{\lambda_b}\nu^N\to 0,\,\, N\to \infty
\end{multline*}
Recalling that $d^{wu}(f^{-N}(a), f^{-N}(c))\approx 1$ we conclude that $W^{wu}$-holonomy is not Lipschitz inside $\cT$, which yields a contradiction.

A very similar argument works for showing that the smaller multiplier $\mu_p$ is a constant independent of periodic point $p$. We illustrate it in the bottom half of Figure~3. Define $\mu_-=\inf_p\mu_p$ and $\mu_+=\sup_p\mu_p$. Then if $\mu_+>\mu_-$ we can obtain a contradiction is a very similar way. Namely we will pick periodic points $a$ and $b$ such that
$$
\mu_-\lesssim \mu_a<\mu_b-\delta<\mu_b
$$
Then the points $\tilde a$, $c$ and $\tilde c$ can be arranged in the same way as before, but requiring that $\tilde a$ is on the ``weakest" invariant submanifold of $f^{n_a}\colon W^{wu}_{f,loc}\to W^{wu}_{f,loc}$ corresponding to $\mu_a$, so that
$$
d^{su}(f^{-i}(a),f^{-i}(\tilde a))\ge C\mu_a^{-1}d^{su}(a,\tilde a)
$$
(In the case when $\mu_a=\lambda_a$ there is no ``weakest'' submanifold and $\tilde a$ does not have to be chosen in any special way in order to have desired control.) Using the same estimates as before we would have that the ratio
$$
\frac{d^{su}(f^{-N}(a), f^{-N}(\tilde a))} {d^{su}(f^{-N}(c), f^{-N}(\tilde c))}
$$
is arbitrarily large contradicting the Lipschitz property again. As we mentioned already $\lambda_p$ and $\mu_p$ being constant implies that  the Jacobian $J^{su}$ is cohomologous to a constant.
\end{proof}

\subsection{Case II: when $E^1_f\nsubset E$} In this case let $j$ be the smallest index such that $E_f^j\subset E$. Such $j$ exists because $E$ is a non-trivial subbundle. If we denote by $\cF^j_f$ the integral foliation of $E_f^j$ then, from definition of $j$ we have $\cF^j_f= W\cap W^{j,u}_f$. Because both foliations $W$ and $W^{j,u}_f$ integrate jointly with $W_f^s$, so does their intersection $\cF^j_f$.

\begin{lemma}
\label{lemma_linear}
Let $H_f$ be the conjugacy to the linear model, $H_f\circ f=L\circ H_f$. If $\cF^j_f$ integrates jointly with $W^s_f$ then $H(\cF^j_f)$ is an $L$-invariant linear minimal foliation on $\T^d$.
\end{lemma}
\begin{proof}
Denote by $W^s_f\oplus\cF^j_f$ the foliation to which $W_f^s$ and $\cF_f^j$ integrate jointly. Then $H(W^s_f\oplus\cF^j_f)$ is an $L$-invariant $C^0$ foliation which is sub-foliated by a minimal linear foliation $W^s_L=H(W^s_f)$. Then $H(W^s_f\oplus\cF^j_f)$ must be linear as well by~\cite[Lemma 2.1]{GS} (which is a higher dimensional generalization of a lemma from~\cite{RGZ}). Now we can write $H(\cF^j_f)$ as an intersection of two linear foliations 
$H(\cF^j_f)=H(W^s_f\oplus\cF^j_f)\cap W^s_L$. Hence $H(\cF^j_f)$ is linear. Minimality follows easily from irreducibility of $L$.
\end{proof}

Now consider the integral foliation of $E_f^{j-1}\oplus E_f^j$ which we denote by $\cT$ by analogy with Case~I. Then, by the same observation which we used in Case~I, we have that $C^1$ regularity of $W$ inside $W^u_f$ implies $C^1$ regularity of the foliation $\cF^j_f$ when restricted to $\cT$.

\begin{lemma}
\label{lemma_const2} If $\cF^j_f$ is $C^1$ inside $\cT$ and $H(\cF^j_f)$ is linear then $J^{j-1}_f$ is cohomologous to a constant.
\end{lemma}
The proof of this lemma is exactly the same as the proof of Lemma~\ref{lemma_lipschitz}. Indeed we have all the properties on which the proof builds upon:
\begin{enumerate}
 \item $\cF^j_f$  and $\cF^{j-1}_f$  are transverse sub-foliations of $\cT$;
 \item  $\cF^j_f$ is conjugated to a minimal linear $L$-invariant foliation by Lemma~\ref{lemma_linear}; 
 \item $\cF^j_f$ is $C^1$ in $\cT$.
 \end{enumerate}
 The only difference compared to the setting og Lemma~\ref{lemma_lipschitz} is the Lipschitz foliation is not the weak foliation in $\cT$, but the strong foliation $\cF^j_f$, however which foliation is faster was irrelevant for the proof of Lemma~\ref{lemma_lipschitz}.

Hence $J^{j-1}_f$ is cohomologous to a constant which means that $f$ is does not belong to the subset $\cV$, ruling out Case II as well, and, thus, finishing the proof of Theorem~\ref{thm_codim1}.

\subsection{Proof of Addendum~\ref{add}} The strategy for the proof is exactly the same but we cannot work with the full Jacobian anymore because it is cohomologous to a constant for volume preserving diffeomorphisms. Instead we can work with the stable Jacobian.

We define $\cU'$ to consists of volume preserving Anosov diffeomprhisms $f$ which are sufficiently $C^1$ close to $L$ so that they retain the dominated splitting and such that $b^s(f)>1$, $b^u(f)>1$, and such that $\mu^{b^s(f)-1}>\xi_l$.

Recall that $J^jf$ denotes the Jacobian of $Df|_{E_f^j}$. Define an open dense subset $\cV'$ in the following way.
$$
\cV'=\{f\in\cU':  \log J^sf, \log J^1f, \ldots \log J^lf\,\,\,\, \mbox{are not cohomologous to constants}\}
$$ 
Now we apply Theorem~\ref{thm_tech2} to diffeomorphisms $f_1$ and $f_2$ with $\kappa=\min\{b^s(f_1)-1, b^s(f_2)-1\}$. Because of our choice of $\kappa$ and because $b^u(f_i)>1$ and $\mu>\xi_l$ the assumption of Theorem~\ref{thm_tech2} boils down to
$$
\mu^{b^s(f_i)-1}>\xi_l
$$
which is satisfied according to our choice of the neighborhood $\cU'$.

Now note that $J^sf_i$ are $C^{b(f_i)}$ regular with $b(f_i)\ge 1+\kappa$, and by matching of stable Jacobians we have $(f_1, J^sf_1)\sim (f_2, J^sf_2)$. Hence item~6 of Theorem~\ref{thm_tech2} applies to $J^sf_i$. These Jacobians are not cohomologous to a constant by our choice of $\cV'$ and we can proceed from here in exactly the same way as in the proof of Theorem~\ref{thm_codim1}.

\end{document}